\documentclass[reqno,12pt]{amsart}
\usepackage{amscd,amsfonts,mathrsfs,amsthm,amssymb, amsmath,mathtools}
\usepackage{upgreek}
\usepackage{csquotes}
\usepackage{enumerate}
\usepackage{bbm}
\usepackage{multicol}
\usepackage{stmaryrd,color}
\usepackage{array}
\usepackage{ae}
\usepackage{accents}
\usepackage{epsfig}
\usepackage[e]{esvect}
\usepackage[nobysame]{amsrefs}
\usepackage[margin=1in]{geometry}
\usepackage[toc,title,page]{appendix}
\usepackage{hyperref}
\hypersetup{colorlinks=true,linkcolor=blue,filecolor=blue,urlcolor=blue, citecolor=blue}

\let\oldtocsection=\tocsection
\let\oldtocsubsection=\tocsubsection
\let\oldtocsubsubsection=\tocsubsubsection
\renewcommand{\tocsection}[2]{\hspace{0em}\oldtocsection{#1}{#2}}
\renewcommand{\tocsubsection}[2]{\hspace{1em}\oldtocsubsection{#1}{#2}}
\renewcommand{\tocsubsubsection}[2]{\hspace{2em}\oldtocsubsubsection{#1}{#2}}

\def\equationcolor {\color{black}}
\def\textcolor     {\color{black}}

\def\bcoleq    {\begin{equation}\equationcolor}
\def\ecoleq    {\textcolor\end{equation}}
\def\bcoleqn   {\equationcolor\begin{eqnarray}}
\def\ecoleqn   {\end{eqnarray}\textcolor}

\def\C        {\mathbb{C}}
\def\S        {\mathbb{S}}

\def\h{{hor}}
\def\v{{ver}}

\def\rim{{{R}_M}}

\def\Re{\operatorname{Re}}
\def\Im{\operatorname{Im}}

\def\span{\operatorname{span}}

\def\div{\operatorname{div}}
\def\det{\operatorname{det}}
\def\R{\mathbb{R}}

\DeclareMathOperator*{\trace}{trace}

\newtheorem{theorem}{Theorem}[section]
\newtheorem{mythm}{Theorem}

\newtheorem{lemma}[theorem]{Lemma}

\newtheorem*{thma}{Theorem A}
\newtheorem*{thmb}{Theorem B}
\newtheorem*{thmc}{Theorem C}
\newtheorem*{thmd}{Theorem D}
\newtheorem{corollary}[theorem]{Corollary}
\newtheorem{proposition}[theorem]{Proposition}
\newtheorem{definition}[theorem]{Definition}
\theoremstyle{definition}
\newtheorem*{assumption*}{$\lambda_{1}$-Condition}

\newtheorem{remark}[theorem]{Remark}

\def\pproof#1{\@ifnextchar[\opargproof
{\opargproof[\it Proof of #1.]}}
\def\opargproof[#1]{\par\noindent {\bf #1 }}

\makeatother

\numberwithin{equation}{section}

\begin{document}

\title[Harmonic and minimal vector fields maps]{Gauss maps of harmonic and minimal great circle fibrations}

\author[I. Fourtzis]{\textsc{I. Fourtzis}}
\address{Ioannis Fourtzis\newline
University of Ioannina,
Section of Algebra \& Geometry,
45110 Ioannina, Greece,\newline
{\sl E-mail address:} {\bf i.fourtzis@uoi.gr}
}

\author[M. Markellos]{\textsc{M. Markellos}}
\address{Michael Markellos\newline
University of Ioannina,
Section of Algebra \& Geometry,
45110 Ioannina, Greece,\newline
{\sl E-mail address:} {\bf mmarkellos@hotmail.gr}
}
\author[A. Savas-Halilaj]{\textsc{A. Savas-Halilaj}}
\address{Andreas Savas-Halilaj\newline
University of Ioannina,
Section of Algebra \& Geometry,
45110 Ioannina, Greece,\newline
{\sl E-mail address:} {\bf ansavas@uoi.gr}
}

\renewcommand{\subjclassname}{  \textup{2000} Mathematics Subject Classification}
\thanks{M. Markellos \& A. Savas-Halilaj would like to acknowledge support by (HFRI) Grant No:133.}
\parindent = 0 mm
\hfuzz     = 6 pt
\parskip   = 3 mm
\date{}

\begin{abstract}
We investigate Gauss maps associated to great circle fibrations of $\S^3$. We show that the 
associated Gauss map to such a fibration is harmonic (respectively minimal) if and only if the unit vector field generating the great circle foliation is
harmonic (respectively minimal).
These results can be viewed as analogues of the classical theorem of Ruh \& Vilms about the harmonicity of the Gauss map of a minimal submanifold in the euclidean space. Moreover, we prove that a harmonic or minimal unit vector field in $\S^3$ with great circle integral curves
is a Hopf vector field.
\end{abstract}

\maketitle

\section{Introduction}

Suppose that $(M,g)$ is an $m$-dimensional manifold equipped with a Riemannian metric $g$.
A unit vector field $\zeta\in\mathfrak{X}(M)$ can be regarded as a ``graphical" map from $M$ to its unit tangent bundle $UM$
equipped with the Sasaki metric $g_S$.
There are two natural functionals that we may consider in the space unit vector fields. The first one
is the {\em energy
functional}
$$
E(\zeta)=\frac{m}{2}{\rm vol}(M,g)+\frac{1}{2}\int_M|\nabla\zeta|^2{\rm dvol}_g.
$$
Critical points of the energy functional with respect to variations through nearby unit
vector fields are called {\em harmonic unit vector fields}. The second functional is the {\em volume functional}
$$
V(\zeta)=\int_M \sqrt{{\rm det}(I+(\nabla\zeta)^T\circ(\nabla\zeta))}\,{\rm dvol}_{g}
$$
whose critical points are called {\em minimal unit vector fields}.

One can exploit the complex structure $\mathcal{J}$ of $\C^{n}=\R^n\times\R^n$ to generate interesting unit 
vector fields on the unit sphere $\S^{2n-1}$. More precisely, if $\nu$ is the unit normal of
$\S^{2n-1}$ then $\zeta=\mathcal{J}\nu$ gives rise to a unit vector field whose integral curves are great circles.
A unit vector field of this type is called {\em Hopf vector field} and the 
corresponding quotient map $f:\S^{2n-1}\to\mathbb{CP}^n$ is called {\em Hopf fibration},
where $\mathbb{CP}^n$ stands for the {\em complex projective space} equipped with its 
standard 
Fubini-Study metric. It turns out that the Hopf fibration is a harmonic Riemannian submersion. Additionally, the Hopf fibration is a {\em minimal map}, i.e.  its graph $\varGamma_f:\S^{2n-1}\to\S^{2n-1}\times\mathbb{CP}^n$ is
a minimal submanifold; see \cite{savas-4}. 
Let us point out here that $\mathbb{CP}^1$ is isometric with the
sphere $\S^2(1/2)$
of radius $1/2$. Composing the Hopf fibration with an isometry of $\S^3$ and the homothety
from $\S^2(1/2)$ into $\S^2$, we obtain again a
submersion with totally geodesic fibers. For simplicity, we call all these maps
{\em Hopf fibrations}.

Hopf vector fields are simultaneously harmonic and minimal unit vector fields.
Moreover,
they are the unique global minima of the energy functional restricted to unit
vector fields on the $3$-sphere; see \cite{brito}. Surprisingly, for spheres of dimension higher than 3,
Hopf vector fields are unstable critical points of the energy; see \cite{wiegmink,wood,medrano3}.
Analogously, Hopf
vector fields are absolute volume minimising in their homology classes in $U\S^3$. 
According to a result due to Gluck and Ziller \cite{gluck2}, the converse is also true. Again,
for spheres of dimensions greater than 3, the Hopf vector fields are not minimisers of the volume functional;
see \cite{johnson,pedersen}.

The classification of harmonic or minimal unit vector fields in the sphere $\S^{2n-1}$ as well as
the classification of harmonic or minimal maps $f:\S^{2n-1}\to\mathbb{CP}^n$
is an open problem. For example, in dimension 3, the Hopf vector fields are the only known examples. According to a conjecture of Eells,
any harmonic map $f:\S^3\to\S^2$ must be weakly conformal; see \cite[Note 10.4.1]{baird1} 
or \cite{wangg}. Therefore, an interesting question is whether a harmonic map $f:\S^3\to\S^2$
can be written as the composition of a Hopf fibration with a conformal map from $\S^2$ to
$\S^2$.

In this paper, we will investigate harmonic and minimal unit vector fields in $\S^3$ with totally geodesic integral curves.
According to a beautiful result of Gluck and Warner \cite{gluck3}, any great circle fibration of $\S^3$
generates a graphical surface in $\S^2\times\S^2$. Generically, at least locally,
the converse is also true; see Theorem \ref{localhopf}. In particular, the fibration is smooth and globally defined
in $\S^3$ if and only if the graphical surface is generated by a smooth strictly length decreasing map.
The idea to establish this duality is to consider the map which assigns each great circle of the 
foliation to the $2$-plane of $\R^4$
containing the circle. This mapping may be considered as a {\em Gauss map associated to the great circle fibration}.
Therefore,
there is a huge class of great circle fibrations of $\S^3$. Moreover, the corresponding quotient
maps are homotopic to the Hopf fibration and so they are homotopically non-trivial.
We prove that:
\begin{mythm}\label{thmD}
A unit vector field with totally geodesic integral curves defined in a saturated neighbourhood $V$ of $\S^3$ is minimal if and only if its
corresponding graphical surface is minimal in $\S^2\times\S^2$.
\end{mythm}
Following the terminology introduced by Baird in \cite{baird}, the {\em Gauss map of a submersion}
$f:M\to N$
between oriented Riemannian manifolds of dimensions $m$ and $n$, respectively, is the map
$\mathcal{G}:M\to\mathbb{G}_{m-n}(M)$ which
associates to each point $x\in M$ the tangent plane to the fiber of $f$ passing through 
$x$. Here, $\mathbb{G}_{m-n}(M)$ denotes the Grassmann bundle over $M$, whose
fiber at each point $x\in M$ is the Grassmannian of oriented $(m-n)$-planes in $T_xM$.
Observe that if $g:N\to N$ is a diffeomorphism, then the submersions $g\circ f$ and $f$ 
have the same Gauss map.
If $n=m-1$, then the kernel of $f$ is generated by a unit vector field. Hence, the Gauss map
$\mathcal{G}$ may be regarded as a section of the unit tangent bundle $UM$ of $M$ , i.e. a unit vector field of $M$. Our next result is an analogue of the classical theorem of Ruh
and Vilms \cite{ruh} in the case of great circle fibrations of $\S^3$ and their corresponding
quotient maps. More precisely, we derive the following theorem:

\begin{mythm}\label{thma}
Let $\zeta:V\subset\S^3\to U\S^3$ be a unit vector field whose integral curves are great circles defined in an open
saturated  neighbourhood
$V$ and $f:V\subset\S^3\to\S^2$ the corresponding quotient map. Then $f$ is a harmonic map if and
only if $\zeta:V\to U\S^3$ is a harmonic unit vector field.
\end{mythm}

In the sequel we prove the following result which generalises a previous one in \cite{han}.

\begin{mythm}\label{mythmb}
A harmonic unit vector field $\zeta$ on $\S^3$, whose integral curves are great circles, is a Hopf vector field and the corresponding quotient map $f:\S^3\to\S^2$ is a Hopf fibration.
\end{mythm}

The main ingredient in the proof of Theorem \ref{mythmb} is the (formal) second fundamental form $\varphi$ of the {\em horizontal distribution}, i.e. the perpendicular distribution to the line bundle
spanned by the vector field $\zeta$. It turns out that 
$\varphi$ satisfies a Ricatti type ODE along the
integral curves of $\zeta$, and on
the horizontal distribution it satisfies a Codazzi type system of PDEs; see Section \ref{sec2}. Using the harmonicity of the vector
field, we arrive at the conclusion that the squared mean curvature $(\trace\varphi)^2$
of the horizontal distribution is a subharmonic function.
Then, from the maximum principle and the Bernstein type theorem for strictly length decreasing minimal maps in \cite[Theorem A]{savas1},
we obtain that $\varphi$ is an orthogonal complex structure.
This leads us to the conclusion that $\zeta$ is a Hopf vector field.

Next we turn out attention to minimal unit vector fields on $\S^3$ and prove the following result:

\begin{mythm}\label{mythmc}
A minimal unit vector field $\zeta$ on $\S^3$, whose integral curves are great circles, is a Hopf vector field and the corresponding quotient map $f:\S^3\to\S^2$ is a Hopf fibration.
\end{mythm}

In the proof of Theorem \ref{mythmc}, we use the geometric 
setup of Gluck and Warner \cite{gluck3} to show that if the mean curvature of
$\zeta:(M,\zeta^*g_S)\to(UM,g_S)$ vanishes, then the corresponding graphical surface $G$ in
$\S^2\times\S^2$ is minimal. Again from the Bernstein type theorem
in \cite[Theorem A]{savas1} we obtain that $G$ is totally geodesic, from where we deduce that $\zeta$
is a Hopf vector field.

\section{Geometry of the tangent bundle}
\subsection{The unit tangent bundle} Let us collect here some basic facts about
the geometry of the tangent bundle $TM$ of a manifold $(M,g)$. There is a natural Riemannian
metric on $TM$ whose construction goes back to the seminal paper
of Sasaki \cite{sasaki}. He uses the metric $g$
to construct a metric $g_S$ on $TM$, which nowadays is called the {\em Sasaki metric}.
The construction of $g_S$ is based on a
natural splitting of the tangent bundle $TTM$ of $TM$ into a vertical
and horizontal sub-bundle by means of the Levi-Civita connection $\nabla$
of $g$. Let us briefly recall the construction of the Sasaki metric, following closely the material in
\cite{dombrowski,boeckx,blair}.

Let $(M,g)$ be an $m$-dimensional Riemannian manifold and denote by $\pi:TM\to M$ the canonical projection from
the tangent bundle $TM$ on $M$. The kernel of $d\pi$ gives rise to the {\em vertical distribution}
$\mathscr{V}$ of $TM$, which is smooth and $m$-dimensional, i.e.
$$
\mathscr{V}_{(x,v)}=\ker d\pi_{(x,v)}.
$$
We will use the Levi-Civita connection
$\nabla$ of $g$ to introduce the horizontal distribution on the tangent bundle
 $TM$. This can be
achieved through the {\em connection map} $K:TTM \to TM$ which is defined in the
following way: For fixed vector $X \in T_{(x, v)}TM$ consider a smooth curve
$\gamma(s)=(\alpha(s), v(s))$, $s\in(-\varepsilon,\varepsilon)$,
 in the tangent bundle $TM$ such that
\begin{equation}\label{px}
\gamma(0)=(\alpha(0),v(0))=(x,v)\quad\text{and}\quad\gamma'(0)
=(\alpha'(0),\nabla_{\alpha'(0)}v)=X
\end{equation}
and define
\begin{equation*}
K(X)=K(\gamma'(0))=\nabla_{\alpha'(0)}v.
\end{equation*}
Then we define the {\em horizontal distribution} as the kernel of the connection map $K$, i.e.
$$
\mathscr{H}_{(x,v)}=\ker K_{(x,v)}.
$$
From \eqref{px} we can readily see that $\mathscr{H}_{(x,v)}$ is $m$-dimensional and if $X\in\mathscr{V}_{(x,v)}\cap\mathscr{H}_{(x,v)}$,
then $X=0$. As a matter of fact, at any point $(x,v)\in TM$, we have the following decomposition
$$
T_{(x,v)}TM=\mathscr{V}_{(x,v)}\oplus\mathscr{H}_{(x,v)}.
$$
The Sasaki metric $g_S$ is a Riemannian metric which makes the above sum orthogonal, i.e.
\begin{equation}\label{Eq: Sasaki metric}
g_S(X,Y)=g(d\pi(X),d\pi(Y))+g(K(X),K(Y))
\end{equation}
for any $X,Y\in TTM$. Observe now that the projection map $\pi: (TM,
g_{S})\to (M, g)$ becomes a Riemannian submersion. Any vector $X\in T_{(x,v)}TM$ can be written in the form
$$
X=X^\v+X^\h,
$$
where $X^\v$ stands for the vertical and $X^\h$ the horizontal component of $X$.
Moreover, for any $w\in T_xM$, there exists a
unique
$w^{\h}_{(x,v)}\in \mathscr{H}_{(x,v)}$ and a unique $w^{\v}_{(x,v)}\in \mathscr{V}_{(x,v)}$ such that
$$
d\pi_{(x,v)}(w^{\h}_{(x,v)})=w \quad\text{and}\quad K_{(x,v)}(w^\v_{(x,v)})=w.
$$
In this case, the vector $w^{\v}$ is called the {\em vertical lift} and $w^\h$ the {\em horizontal lift}
of $w$. The horizontal lift of a vector field $w \in \mathfrak{X}(M)$ is
the vector field $w^\h\in\mathfrak{X}(TM)$ which value at a point $(x,v)$ is the
horizontal lift of $w_x$ to $(x, v)$. The vertical lift $w^\v$ of the vector field $w$ is
defined similarly. There is also a natural complex structure $J$ on $TTM$, i.e. define
$$
Jw^\h=w^\v \quad\text{and}\quad Jw^\v=-w^\h,
$$
for any $w\in\mathfrak{X}(M)$.

Consider now the {\em unit tangent bundle} $UM$ of $M$, i.e.
$
UM=\big\{(x,v)\in TM:g(v,v)=1\big\}.
$
The unit tangent bundle $UM$ is an embedded hypersurface of $TM$ and we can equip it
with the induced from $g_S$ Riemannian metric. One can readily check that the vector field
$\eta_{(x, v)}=v^{\v}_{(x,v)}$ is
a unit normal of $UM$. Note that if $w\in\mathfrak{X}(M)$, then
$$
g_S\big(w^\h_{(x,v)},\eta_{(x,v)}\big)=g_S\big(w^\h_{(x,v)},v^{\v}_{(x,v)}\big)=0.
$$
Consequently, the horizontal lift $w^{\h}$ of a vector field $w\in\mathfrak{X}(M)$ is always tangent
to $UM$. On the other hand, the vertical lift $w^{\v}$ is not necessarily tangent to $UM$. Following
the terminology introduced by Boeckx and Vanhecke
\cite{boeckx} we call {\em tangential lift} $w^{tan}$ of $w$ the tangential to $UM$ component of
$w^{\v}$, i.e.
\begin{equation*}
w^{tan}_{(x, v)}=w^{\v}_{(x,v)}-g_S\big(w^{\v}_{(x,v)},v^{\v}_{(x,v)}\big)\eta_{(x,v)}
=w^{\v}_{(x, v)}-g(w, v)\eta_{(x, v)}.
\end{equation*}
At the point $(x, v)\in UM$, the tangent space of $UM$ can be written as
\begin{equation*}
T_{(x, v)}UM=\big\{w_1^{hor}+w_2^{tan}: w_1, w_2 \in T_{x}M\big\}
=\big\{w_1^{hor}+w_2^{ver}: w_1, w_2 \in T_{x}M,\, g(w_2, v)=0\big\}.
\end{equation*}

On  the unit tangent bundle $UM$ there is an analogue of the Hopf vector field defined from the complex structure
of $TM$. More precisely, the vector field $\xi$ given by
$$\xi_{(x, v)}=-J(\eta_{(x,v)})=v^{\h}_{(x,v)},$$
for $(x, v)\in UM$, is unit, tangent along the unit tangent bundle $UM$ and is called the {\em geodesic flow vector field}.

The next formula will be used later and can be found in \cite[page 176]{blair} or in \cite[page 82]{boeckx}.
\begin{lemma}\label{boeckx}
For any $X,Y\in\mathfrak{X}(M)$, the following formula hold
\begin{equation*}
\nabla^{g_{S}}_{X^{\h}}Y^{\h}=(\nabla_{X}Y)^{hor}-\tfrac{1}{2}\big(\rim(X, Y)v \big)^{tan},
\end{equation*}
where $\rim$ stands for the Riemann curvature operator of the metric $g$.
\end{lemma}
\subsection{Critical points of the energy and volume functional}

A unit vector field $\zeta$ on $(M, g)$ can be regarded as the embedding
$\zeta: M\to UM$ given by $\zeta(x)=(x, \zeta_{x}).$
Let us collect in the proposition below the Euler-Lagrange equations for the critical points of the energy
and volume functional as well as some useful identities. For the proofs we refer to \cite{wood,wiegmink,medrano2,yampolsky}.

\begin{proposition}
Let $\zeta:M\to UM$ be a unit vector field, which we view as a mapping of the form
$\zeta(x)=(x,\zeta_x),$
for any $x\in M$.
\begin{enumerate}[\rm(a)]
\item
The differential of $\zeta$ is given by
\begin{equation}\label{Eq: differential of zeta}
d\zeta(X)=X^{hor}+(\nabla_{X}\zeta)^{tan},
\end{equation}
for all vector fields $X\in\mathfrak{X}(M)$. Moreover, we have that
$$(\nabla_{X}\zeta)^{tan}=(\nabla_{X}\zeta)^{ver}.$$
for all vector fields $X\in\mathfrak{X}(M)$.
\medskip
\item
The pull-back via $\zeta$ Riemannian metric $\zeta^{*}g_{S}$ on $M$ satisfies
\begin{equation}\label{Eq: pull back metric Sasaki}
(\zeta^{*}g_{S})(X, Y)=g_{S}(d\zeta(X), d\zeta(Y))=g(X, Y)+g(\nabla_{X}\zeta, \nabla_{Y}\zeta),
\end{equation}
for all $X,Y\in\mathfrak{X}(M)$.
\medskip
\item It holds that $g=\zeta^*g_S$ if and only if the vector field $\zeta$ is
parallel.
\medskip
\item
The vector field $\zeta$ is a critical points of the energy functional with respect to variations through nearby unit vector fields if and only if
$$
\Delta\zeta+|\nabla\zeta|^2\zeta=0
$$
where $\Delta$ is the rough Laplacian.
\medskip
\item
Consider the $(1,1)$-tensor $\varphi$ on $M$ given by the formula
$$
\varphi(X)=-\nabla_X\zeta
$$
for any $X\in\mathfrak{X}(M)$. Then:
\begin{enumerate}
\medskip
\item[{\rm (e$_1$)}]
At any $x\in M$, there exist orthonormal frames
$\{v_0,v_1,\dots,v_{m-1}\}$ and $\{\beta_1,\dots\beta_{m-1}\}$ with respect to $g$ such that
$$
\varphi(v_0)=0,\,\,\phi(v_1)=\lambda_1\beta_1,\,\,\dots,\,\,\varphi(v_{m-1})=\lambda_{m-1}\beta_{m-1},
$$
where
$$\lambda_1\ge\dots\ge\lambda_{m-1}\ge 0.$$ The numbers $\lambda_i$, $i\in\{1,\dots,m\}$,
are the singular values of $\varphi$ at the point $x\in M$.
\medskip
\item[{\rm (e$_2$)}]
With the above ordering, the singular values give rise to
continuous  functions on $M$. Moreover, they are smooth on an open and dense subset of $M$.
In particular, they are smooth on open subsets where the corresponding multiplicities are constant
and the corresponding  eigenspaces are smooth distributions.
\end{enumerate}
\medskip
\item
The vectors
$$
e_0=v_0,\,\,e_1=\frac{v_1}{\sqrt{1+\lambda_1^2}},\,\,\dots,\,\,
e_{m-1}=\frac{v_{m-1}}{\sqrt{1+\lambda_{m-1}^2}},
$$
form an orthonormal basis at the point $x\in M$ with respect to the induced Riemannian metric
$\zeta^*g_S$. Moreover, the vectors
$$\xi_1=\frac{-\lambda_1e_1^\h+\beta_1^\v}{\sqrt{1+\lambda_1^2}},\,\,\dots,\,\,
\xi_{m-1}=\frac{-\lambda_{m-1}e_{m-1}^\h+\beta_{m-1}^\v}{\sqrt{1+\lambda_{m-1}^2}},
$$
form an orthonormal basis of the normal bundle of the embedding $\zeta$ at $\zeta(x)$.
\medskip
\item
The vector field $\zeta$ is a critical point of the volume functional if and only if the isometric embedding
$\zeta:(M,\zeta^*g_S)\to (UM,g_S)$ is minimal. 
\medskip
\item
The mean curvature vector $H_\zeta$ of the embedding $\zeta$ is given by the formula
$$
H_\zeta=\sum_{\alpha=1}^{m-1}\frac{1}{\sqrt{1+\lambda_\alpha^2}}\left\{\langle \nabla^2_{e_0,e_0}\zeta,\beta_\alpha\rangle
+\sum_{i=1}^{m-1}\frac{\langle \nabla^2_{e_i,e_i}\zeta,\beta_\alpha\rangle-\lambda_\alpha\lambda_i
R_M(e_\alpha,e_i,\zeta,\beta_i)}{1+\lambda^2_{i}}\right\}\xi_\alpha,
$$
where $R_M$ stands for the Riemannian curvature operator of $g$ and
$$
\nabla^2_{v,w}\zeta=\nabla_v\nabla_w\zeta-\nabla_{\nabla_vw}\zeta, \quad v,w\in\mathfrak{X}(M),
$$
is the Hessian of the vector field $\zeta$; compare with Corollary 1(b) in \cite{savas-4}. 
\end{enumerate}
\end{proposition}

\section{Fibrations with totally geodesic fibers}\label{sec2}

\subsection{Great circle fibrations and the Ricatti equation}
Let $V$ be an open and convex subset of $\S^3$ and $f:V\to\S^2$ be a submersion with totally geodesic fibers.
Denote by $\zeta$ the unit vector field generating the fibers, by $\mathcal{V}=\ker df=\span\{\zeta\}$ the {\em vertical
line bundle} and by $\mathcal{H}=\mathcal{V}^\perp$ the {\em horizontal plane bundle}. Consider the tensor $\varphi: \mathcal{H} \to \mathcal{H}$,
given by
$$\varphi(v)=-\nabla_{v}\zeta,$$
for all $v\in \mathcal{H}$, where $\nabla$ stands for the standard Levi-Civita connection of $\S^3$. The tensor
$\varphi$ is the (formal) {\em second fundamental form} of $\mathcal{H}$.
It is well-known that $\varphi$ satisfies the equations
\begin{equation}\label{ricatti}
\nabla^{\mathcal{H}}_{\zeta}\varphi=\varphi^2+I \quad\text{and}\quad (\nabla^{\mathcal{H}}_{v}\varphi)w
-(\nabla^{\mathcal{H}}_{w}\varphi)v=0,
\end{equation}
for any pair of vector fields $v,w$ on $\mathcal{H}$; see for example \cite[page 313]{baird1}.
Denote by $J$ the complex structure of $\mathcal{H}$ and by
$\{\alpha_1=\zeta,\alpha_2,\alpha_3=J\alpha_2\}$ a local orthonormal frame on $V$. Moreover, denote by
$\varphi_{ij}=\langle\varphi(\alpha_i),\alpha_j\rangle$, $i,j\in\{2,3\}$, the components of $\varphi$ with respect
to the orthonormal frame
$\{\alpha_2,\alpha_3\}$. Then,
$$\trace\varphi=\varphi_{22}+\varphi_{33}=-\div(\zeta).$$
From the Ricatti equation in \eqref{ricatti}, we obtain the  following:
\begin{lemma}\label{ricattieqns}
Let $\gamma$ be an integral curve of $\zeta$ and $\{\alpha_2, \alpha_3\}$ is a parallel orthonormal frame 
along $\gamma^*\mathcal{H}$. Then:
\begin{enumerate}[\rm (a)]
\item
The components $\varphi_{ij}$, $i,j\in\{2,3\}$, of $\varphi$ with respect to this frame, satisfy the ODEs
\begin{equation}\label{ric1}
\left\{
\begin{array}{lll}
\zeta(\varphi_{22})&=&1+\varphi_{22}^2+\varphi_{23}\varphi_{32},\\
\zeta(\varphi_{33})&=&1+\varphi_{33}^2+\varphi_{23}\varphi_{32}, \\
\zeta(\varphi_{23})&=&\varphi_{23}(\varphi_{22}+\varphi_{33}), \\
\zeta(\varphi_{32})&=&\varphi_{32}(\varphi_{22}+\varphi_{33}).
\end{array}
\right.
\end{equation}
\item
The functions $\trace{\varphi}$, $\trace{(\varphi\circ J})$ and $\det\varphi$ satisfy the ODEs
\begin{equation}\label{ric2}
\left\{
\begin{array}{lll}
\zeta(\trace(\varphi\circ J))&=&(\trace\varphi)(\trace(\varphi\circ J)),\\
\zeta(1+\det \varphi)&=&(\trace\varphi)(1+\det \varphi), \\
\zeta(\trace\varphi)&=&(\trace\varphi)^2-2(1+\det \varphi)+4,
\end{array}
\right.
\end{equation}
and
\begin{equation}\label{ric3}
\zeta\big\{(\trace\varphi)^2-4\det \varphi\big\}=2(\trace\varphi)\big\{(\trace\varphi)^2-4\det \varphi\big\}.
\end{equation}
\end{enumerate}
\end{lemma}

\subsection{Great circle fibrations and length decreasing maps}
Due to an impressive result of Gluck and Warner \cite{gluck3}, there is a relation between great circle fibrations of the $3$-sphere
and length decreasing maps between two dimensional euclidean spheres. This construction is obtained as follows: Associate at each point $x\in\S^3$
the $2$-dimensional subspace of $\R^4$ spanned by the great circle of the fibration passing through $x$. In this way, we obtain a map with values in the
Grassmann space $\mathbb{G}_2(\R^4)\simeq\S^2\times\S^2$ of oriented 2-planes is $\R^4$. It turns out that, the image of this particular map is a two
dimensional surface in $\S^2\times\S^2$ which is the graph of a
strictly length decreasing map. The converse is also true, i.e. any strictly length decreasing map between $2$-dimensional euclidean unit spheres, gives rise to a great circle foliation of the $3$-sphere.
To make the paper self-contained, let us describe here this duality following our exposition.

Denote by $\Lambda^2(\R^4)$ the dual space of all alternative multilinear forms of
degree $2$. Elements of $\Lambda^2(\R^4)$  are called {\em $2$-vectors.} As a matter of fact,
for given vectors $v_1$ and $v_2$ on $\R^4$, the {\em exterior product}
$v_1\wedge v_2$
is the linear map which on an alternating form $\Omega$ of degree $2$ takes the value
$$
(v_1\wedge v_2)(\Omega)=\Omega(v_1,v_2).
$$
The exterior product is linear in each variable separately. Interchanging
two elements the sign of the product changes and if two variables are the same
the exterior product vanishes. A $2$-vector $\omega$ is called {\em simple} or {\em decomposable}
if it can be
written as a single wedge product of vectors, that is
$$
\omega=v_1\wedge v_2.
$$
Note that there are $2$-vectors which are not simple. One can verify that
the exterior product $v_1\wedge v_2$ is zero if and only if the
vectors are linearly dependent. Moreover, if $\{\varepsilon_1,\varepsilon_2,\varepsilon_3,\varepsilon_4\}$
consists a basis for $\R^4$, then the collection
$\{\varepsilon_i\wedge\varepsilon_j:1\le i< j\le 4\}$
consists a basis of $\Lambda^2(\R^4)$. Therefore, the dimension of the
vector space of $2$-vectors is $6$.

We can equip $\Lambda^2(\R^4)$ with a natural inner product $(\cdot\,,\cdot)$. Indeed, define
$$
(v_1\wedge v_2,w_1\wedge w_2)=
\langle v_1,w_1\rangle\langle v_2,w_2\rangle-\langle v_1,w_2\rangle\langle v_2,w_1\rangle,
$$
on simple $2$-vectors and then extend linearly. Note that, if $\{\varepsilon_1,\varepsilon_2,\varepsilon_3,\varepsilon_4\}$ is an
orthonormal basis of $\R^4$ then, the $2$-vectors $\{\varepsilon_i\wedge\varepsilon_j:1\le i< j\le 4\}$
consists an orthonormal basis for the exterior power $\Lambda^2(\R^4)$.

Each simple vector represents a unique $2$-dimensional subspace of $\R^4$. Moreover,
if $\omega_1$ and $\omega_2$ are simple vectors representing the same subspace, then there exists
a non-zero real number $\lambda$ such that
$\omega_1=\lambda\,\omega_2.$
Therefore, there is an obvious equivalence relation
on the space of simple $2$-vectors such that the space of equivalence classes is to an one-to-one correspondence
with the space of $2$-dimensional subspaces of $\R^4$.
Additionally, we can consider another relation
on the set of non-zero simple $2$-vectors: {\em $\omega_1$ and $\omega_2$ are called equivalent
if and only if
$\omega_1=\lambda\,\omega_2$
for some positive number $\lambda$.} Denote by $[\omega]$
the class containing all simple $2$-vectors that are equivalent to $\omega$.
The equivalence
classes now obtained are called {\em oriented $2$-dimensional subspaces} of $\R^4$
and the space $\mathbb{G}_2(\R^4)$ of all equivalence classes is called {\em Grassmann space}
of oriented $2$-planes of $\R^4$. Consequently, a plane $\varPi$ in $\R^4$ can be associated
with the equivalence class of the $2$-vector $\omega=v_1\wedge v_2$, where $\{v_1,v_2\}$ is an orthonormal
basis of $\varPi$.

There exists a natural linear endomorphism $\ast$ of $\Lambda^2(\R^4)$ which maps a
$2$-plane $\varPi$ in $\R^4$ into its orthogonal complements $\varPi^{\perp}$. Specifically,
if $\{\varepsilon_1,\varepsilon_2,\varepsilon_3,\varepsilon_4\}$ is the standard orthonormal basis
of $\R^4$ and
$$\omega=\alpha_{12}\varepsilon_{1}\wedge \varepsilon_{2}+
\alpha_{13}\varepsilon_{1}\wedge \varepsilon_{3}+\alpha_{14}\varepsilon_{1}\wedge \varepsilon_{4}
+\alpha_{23}\varepsilon_{2}\wedge \varepsilon_{3}+\alpha_{24}\varepsilon_{2}\wedge \varepsilon_{4}
+\alpha_{34}\varepsilon_{3}\wedge \varepsilon_{4},$$
we define
$$*\omega=\alpha_{34}\varepsilon_{1}\wedge \varepsilon_{2}-\alpha_{24}\varepsilon_{1}\wedge \varepsilon_{3}
+\alpha_{23}\varepsilon_{1}\wedge \varepsilon_{4}+\alpha_{14}\varepsilon_{2}\wedge \varepsilon_{3}
-\alpha_{13}\varepsilon_{2}\wedge \varepsilon_{4}+\alpha_{12}\varepsilon_{3}\wedge \varepsilon_{4}.$$
The operator $\ast$ is called {\em Hodge star operator}. Let us mention here that $\ast$ is an
isometry and it satisfies
$$\ast^2=\ast\circ\ast=I.$$
Using elementary arguments, one can show that a non-zero $2$-vector $\omega$ is simple if and only
if $\omega\wedge\omega = 0$ or, equivalently, if and only if $(\omega,\ast\omega)=0$. Hence, we may represent
the space of oriented two planes in $\R^4$ in the form
$$
\mathbb{G}_2(\R^4)=\big\{[\omega]:\omega\in\Lambda^2(\R^4),\,\,\|\omega\|_{\Lambda^2(\R^4)}=1\,\,\text{and}\,\,(\omega,\ast\omega)=0\big\}.
$$
The Hodge star operator has eigenvalues $+1$ and
$-1$, both of multiplicity $3$. In particular, the corresponding eigenspaces of the
$\ast$ are $3$-dimensional and they are given
by
$$
E_{-}=\{ \omega\in\Lambda^{2}(\mathbb{R}^{4}): *\omega=-\omega\big\}
\quad\text{and}\quad
E_{+}=\big\{ \omega\in\Lambda^{2}(\mathbb{R}^{4}): *\omega=\omega\}.
$$
The eigenspaces $E_-$ and $E_+$ are mutually perpendicular and
$$
\Lambda^{2}(\mathbb{R}^{4})=E_{-}\oplus E_{+}.
$$
As a matter of fact, any element $\omega\in\Lambda^2(\R^4)$ can be uniquely written in the form
\begin{equation}\label{dirsum}
\omega=\frac{\omega-*\omega}{2}\oplus\frac{\omega+*\omega}{2},
\end{equation}
where the first term in the right hand side of \eqref{dirsum} belongs to $E_-$ and the second to
$E_+$.
Moreover,
if $\{\varepsilon_1,\varepsilon_2,\varepsilon_3,\varepsilon_4\}$ is the standard basis of $\R^4$,
then the collection
$$\left\{ \frac{\varepsilon_{1}\wedge \varepsilon_{2}+\varepsilon_{3}\wedge \varepsilon_{4}}{2},
\frac{\varepsilon_{1}\wedge \varepsilon_{3}-\varepsilon_{2}\wedge \varepsilon_{4}}{2},
\frac{\varepsilon_{1}\wedge \varepsilon_{4}+\varepsilon_{2}\wedge \varepsilon_{3}}{2}\right\},$$
forms an orthogonal basis of $E_{+}$ and
$$\left\{ \frac{\varepsilon_{1}\wedge \varepsilon_{2}-\varepsilon_{3}\wedge \varepsilon_{4}}{2}, \frac{\varepsilon_{1}\wedge \varepsilon_{3}+\varepsilon_{2}\wedge \varepsilon_{4}}{2},
\frac{\varepsilon_{1}\wedge \varepsilon_{4}-\varepsilon_{2}\wedge \varepsilon_{3}}{2}\right\},$$
forms an orthogonal basis of $E_{-}$.

Consider now the euclidean spheres
$$\mathbb{S}^{2}_{-}=\big\{ \omega\in E_{-}: \|\omega\|_{\Lambda^{2}(\mathbb{R}^{4})}=
1/\sqrt{2}\big\}\quad \text{and} \quad
\mathbb{S}^{2}_{+}=\big\{ \omega\in E_{+}: \|\omega\|_{\Lambda^{2}(\mathbb{R}^{4})}=1/\sqrt{2}\big\}.$$
\begin{lemma}The Grassmann space $\mathbb{G}_{2}(\mathbb{R}^{4})$ can be identified with the direct
product $\mathbb{S}^{2}_{-}\times\mathbb{S}^{2}_{+}$.
\end{lemma}
\begin{proof}
Let $\varPi=[\omega]\in\mathbb{G}_{2}(\mathbb{R}^{4})$ and assume that the representative $\omega$ is chosen to have the form $\omega=v_1\wedge v_2,$
where $\{v_1,v_2\}$ is an orthonormal basis of the under consideration plane. Observe that, if $\{w_1,w_2\}$ is another orthonormal frame of $\varPi$ which belongs in the same orientation with $\{v_1,v_2\}$,
then
$$v_1\wedge v_2=w_1\wedge w_2.$$
Now, the map
\begin{equation}\label{dirsum1}
[\omega]\longmapsto\frac{\omega-*\omega}{2}\oplus\frac{\omega+*\omega}{2}
\end{equation}
is well defined and gives rise to a bijection between $\mathbb{G}_{2}(\mathbb{R}^{4})$ and $\S^2_-\times\S^2_+$. This completes the proof.
\end{proof}

Suppose now that $V$ is a saturated open neighbourhood of $\S^3$ such that $f(V)\subset\S^2_-$ is simply connected.
Define the maps $h_{\pm}:f(V)\subset\S^2_-\to\S^2_\pm$ given by
\begin{equation}\label{graph}
h_{\pm}\big(f(x)\big)=\frac{x\wedge \zeta(x)\pm*(x\wedge \zeta(x))}{2},
\end{equation}
where $x\in V$. We give now an alternative proof of a beautiful result first proved by Gluck and Warner \cite{gluck3}.
\begin{lemma}\label{gw}
The following statements hold true:
\begin{enumerate}[\rm(a)]
\item
The maps $h_{\pm}$ are well-defined and smooth.
\medskip
\item
The map $G=(h_-,h_+)$ gives rise to a strictly length decreasing graphical map.
\medskip
\item
If $h_-$ is a diffeomorphism, then
the singular values $\mu_1$ and $\mu_2$
of $h_+\circ h^{-1}_-$ are related with the second fundamental form $\varphi$ by
$$
\mu_1\circ f=\frac{\Big|\sqrt{(1-\det \varphi)^{2}+(\trace \varphi)^{2}}
-\sqrt{|\varphi|^{2}-2\det\varphi}\,\Big|}{1+\det\varphi+\trace{(\varphi\circ J})}
$$
and
$$
\mu_2\circ f=\frac{\sqrt{(1-\det \varphi)^{2}+(\trace \varphi)^{2}}+\sqrt{|\varphi|^{2}-2\det\varphi}}{1+\det\varphi+\trace{(\varphi\circ J})}.
$$
\end{enumerate}
\end{lemma}
\begin{proof}
(a) Consider points $x,y\in V$ such that $f(x)=f(y)$. Then, $x$ and $y$ belongs in the same circle of the foliation from where we deduce that
$$x\wedge \zeta(x)=y\wedge \zeta(y).$$
Therefore,
$$h_{\pm}(f(x))=h_{\pm}(f(y))$$
and $h_{\pm}$ is well defined. Smoothness of the maps $h_{\pm}$ is clear.

(b) Let us compute now the differentials of $h_{\pm}\circ f$. Fix a point $x\in V$ and consider a local orthonormal
basis $\{\alpha_1=\zeta,\alpha_2,\alpha_3\}$ of $T_xV$. Then,
\begin{equation*}
\begin{array}{lll}
\ast (x\wedge \alpha_{1}) =\,\,\,\, \alpha_{2}\wedge \alpha_{3},&\ast (x\wedge \alpha_{3})  \,\,\,=\alpha_{1}\wedge \alpha_{2},&\ast (\alpha_{1}\wedge \alpha_{3}) =-x\wedge \alpha_{2},\\
\ast (x\wedge \alpha_{2}) =-\alpha_{1}\wedge \alpha_{3},&\ast (\alpha_{1}\wedge \alpha_{2}) =x\wedge \alpha_{3},&\ast (\alpha_{2}\wedge \alpha_{3}) = \,\,\,\,x\wedge \alpha_{1}.
\end{array}
\end{equation*}
Differentiating \eqref{graph} with respect to $\alpha_2$ and $\alpha_3$, we get that
\begin{eqnarray}
d(h_-\circ f)(\alpha_2)&=&\,\,\,\,(1-\varphi_{23})\frac{x\wedge \alpha_3-\alpha_1\wedge \alpha_2}{2}-\varphi_{22}\frac{x\wedge \alpha_2+\alpha_1\wedge \alpha_3}{2},\label{dif1}\\
d(h_-\circ f)(\alpha_3)&=&-\varphi_{33}\frac{x\wedge \alpha_3-\alpha_1\wedge \alpha_2}{2}
-(1+\varphi_{32})\frac{x\wedge \alpha_2+\alpha_1\wedge \alpha_3}{2},\label{dif2}\\
d(h_+\circ f)(\alpha_2)&=&-(1+\varphi_{23})\frac{x\wedge \alpha_3+\alpha_1\wedge \alpha_2}{2}-\varphi_{22}\frac{x\wedge \alpha_2-\alpha_1\wedge \alpha_3}{2},\label{dif3}\\
d(h_+\circ f)(\alpha_3)&=&-\varphi_{33}\frac{x\wedge \alpha_3+\alpha_1\wedge \alpha_2}{2}
+(1-\varphi_{32})\frac{x\wedge \alpha_2-\alpha_1\wedge \alpha_3}{2}.\label{dif4}
\end{eqnarray}
Observe that the vectors
$$
\Big\{\frac{x\wedge \alpha_3-\alpha_1\wedge \alpha_2}{\sqrt 2},\frac{x\wedge \alpha_2+\alpha_1\wedge \alpha_3}{\sqrt 2}\Big\}
$$
form an orthonormal basis of the tangent space at $h_-(f(x))$ of $\S^2_-\subset E_-$ and
$$
\Big\{\frac{x\wedge \alpha_3+\alpha_1\wedge \alpha_2}{\sqrt 2},\frac{x\wedge \alpha_2-\alpha_1\wedge \alpha_3}{\sqrt 2}\Big\}
$$
form an orthonormal basis of the tangent space at $h_+(f(x))$ of $\S^2_+\subset E_+$.

We claim now that either
$$
D_-=\det\begin{bmatrix}
     1-\varphi_{23} & -\varphi_{33}\\
     -\varphi_{22} & -1-\varphi_{32}\\
\end{bmatrix}=-1-\det\varphi+\varphi_{23}-\varphi_{32}=-1-\det\varphi-\trace(\varphi\circ J)\neq 0,
$$
everywhere on $V$, or
$$
D_+=\det\begin{bmatrix}
     -1-\varphi_{23} & -\varphi_{33}\\
     -\varphi_{22} & 1-\varphi_{32}\\
\end{bmatrix}
=-1-\det\varphi+\varphi_{32}-\varphi_{23}=-1-\det\varphi+\trace(\varphi\circ J)\neq 0,
$$
everywhere in $V$. Indeed! From the last two equations of \eqref{ric2} it follows that $1+\det\varphi$ is strictly positive on $V$. Moreover, from the first equation of \eqref{ric2} and the first equation of \eqref{ric1} it follows that
$\varphi_{23}-\varphi_{32}$ is nowhere zero. Consequently, either $\varphi_{23}-\varphi_{32}$ is strictly negative and $D_-$ is strictly negative in $V$, or $\varphi_{23}-\varphi_{32}$ is strictly positive and $D_+$ is strictly
negative in $U$. Since $f(V)$ is simply connected, we deduce that one of the maps $h_{\pm}$ is a diffeomorphism and $G$ is graphical.

Without loss of generality, assume that $h_-$ is a diffeomorphism.
Then, $\varphi_{23}-\varphi_{32}<0$ everywhere in $V$.
In this case, we will show that $h_+\circ h^{-1}_-$ is strictly length decreasing, i.e.
$$|d(h_{+}\circ h_{-}^{-1})(df(a))|< |df(a)|,$$
or, equivalently,
\begin{equation}\label{trionymo}
|d(h_{+}\circ f)(\alpha))|^2< |d(h_{-}\circ f)(\alpha)|^2,
\end{equation}
for any vector $\alpha\in\mathcal{H}$. Indeed! If $\alpha=\kappa_1 \alpha_2+\kappa_2 \alpha_3$, then from \eqref{dif1}, \eqref{dif2}, \eqref{dif3} and \eqref{dif4}, we get that \eqref{trionymo} holds if and only if
\begin{equation}\label{trionymo1}
\varphi_{23}\kappa_1^{2}+(\varphi_{33}-\varphi_{22})\kappa_1\kappa_2-\varphi_{32}\kappa_2^{2}<0,
\end{equation}
for any $\kappa_1,\kappa_2\in\R$. On the other hand, \eqref{trionymo1} holds for any
$\kappa_1,\kappa_2$, if and only if the matrix
$$
A=
\begin{bmatrix}
	\varphi_{23} & \frac{1}{2}(\varphi_{33}-\varphi_{22})\\
	\frac{1}{2}(\varphi_{33}-\varphi_{22}) & -\varphi_{32}\\
\end{bmatrix}
$$
has negative eigenvalues or, equivalently, if and only if
$$
\trace A=\varphi_{23}-\varphi_{32}<0\quad\text{and}\quad
4\det A=-(\trace\varphi)^{2}+4\det\varphi>0.
$$
The validity of the first condition is clear. Suppose now to the contrary, that there is a point $x_0\in V$, where
$$(\trace\varphi)^{2}(x_0)-4\det\varphi(x_0)\geq 0.$$
From \eqref{ric3} it follows that the same inequality holds along the integral curve $\gamma$ of
$\alpha_1$ passing through $x_0$.
From the third identity of \eqref{ric2}, we obtain that along $\gamma$ it holds
\begin{eqnarray*}
\alpha_{1}(\trace\varphi)&=&(\trace\varphi)^{2}-2\det\varphi+2
=(\trace\varphi)^{2}-4\det\varphi+2(\det\varphi+1)\\
&\geq& 2(\det\varphi+1)\\
&>&0,
\end{eqnarray*}
which leads to a contradiction. Therefore, $h_{+}\circ h_{-}^{-1}$ is a strictly length decreasing map.

(c) To compute the singular values of the map $h_{+}\circ h_{-}^{-1}$ we proceed as follows. Fix a point $x_0$ in
$V$ and suppose that $\{\alpha_1=\zeta,\alpha_2,\alpha_3\}\in T_{x}V$
and $\beta_2,\beta_3\in T_{f(x)}\S^2$ are orthonormal basis of the singular decomposition of $f$.
Then, from \eqref{dif1}, \eqref{dif2}, \eqref{dif3} and \eqref{dif4}, we get the expressions for the 
singular values
of $h_+\circ h_-^{-1}$. This completes the proof.
\end{proof}

Let us see now the converse. Suppose that
$G:V\subset\S^2_-\to\S^2_-\times\S^2_+\simeq\mathbb{G}_2(\R^4)$ is the graph of a smooth
map $g:V\subset\S^2_-\to\S^2_+$, where $V$ is an open domain of $\S^2$. Then,
for any $x\in V$, the element
$G(x)=x\oplus g(x)$
describes a plane in $\R^4$ and
$G(x)\cap\S^3$ gives rise to a great circle of $\S^3$. In order to explicitly describe the plane
generated by $G(x)$, we will use the quaternionic structure of $\R^4$. Recall that, as a vector space
the {\em quaternions} are
$$
\mathbb{H}=\{a_0+a_1 i+a_2 j+a_3k:a_0,a_1,a_2,a_3\in\R^{4}\}.
$$
They become an associative algebra with $1$ as the multiplicative unit via
$$
i^2=j^2=k^2=-1,\quad i\cdot j=-j\cdot i=k, \quad j\cdot k =-k\cdot j =i, \quad k\cdot i =-i\cdot k =j.
$$
We denote by $\Re \mathbb{H}$ the one dimensional linear subspace spanned by the element $1$ and $\Im \mathbb{H}$ the orthogonal complement of $\Re \mathbb{H}$. Hence,
an element
$x=x_0+x_1 i+x_2 j+x_3 k\in\mathbb{H}$ can be described by its {\em real
part} ${\rm Re}(x)=x_0$ and its {\em imaginary part}
${\rm Im}(x)=x_1 i+x_2 j+x_3 k$. Moreover, the {\em conjugate} $\overline{x}$
of $x$ is defined to be the quaternionic number
$$
\overline{x}=x_0-x_1 i-x_2 j-x_3 k.
$$
The euclidean inner product and the norm on $\R^4\simeq\mathbb{H}$ can be equivalently written in the 
form
$$\langle x, y\rangle=\Re(x\cdot \overline{y})=\Re(\overline{x}\cdot y)\quad\text{and}\quad
|x|^2=x\cdot \overline{x}=\overline{x}\cdot x,$$
for any $x,y\in\mathbb{H}$. Moreover, the standard outer product in $\R^3$ can be regarded as the
map $\times:{\rm Im}\mathbb{H}\times{\rm Im}\mathbb{H}\to{\rm Im}\mathbb{H}$ given by
$$x\times y={\rm Im}(x\cdot y),$$
for any $x,y\in {\rm Im}\mathbb{H}.$

Let us collect in the following lemma the most important properties
of the quaternionic multiplications; for more details see \cite[page 186]{gluck4}.
\begin{lemma}\label{Eq:Lemmaidentitiesquaternion}
The following identities hold:
\begin{enumerate}[\rm(a)]
\item For any $x,y,z\in\mathbb{H}$, we have
$\langle z\cdot x, z \cdot y\rangle=\langle x, y\rangle|z|^{2}=\langle x\cdot z, y \cdot z\rangle.$
\medskip
\item  For any $x \in \Im \mathbb{H}$, we have $x^{2}=-|x|^{2}.$
\medskip
\item For any $x,y\in\mathbb{H}$, we have $\overline{x\cdot y}=\overline{y}\cdot\overline{x}.$
\medskip
\item  For any $x, y \in \Im \mathbb{H}$, we have
$x\cdot y+y\cdot x=-2\langle x, y\rangle$; hence orthogonal imaginaries anti-commute.
\end{enumerate}
\end{lemma}
Consider now the unit sphere $\S^{3}\subset \mathbb{H}$ as the subset of quaternions of length $1$.
Moreover, we consider $\S^2_{\pm}\subset{\rm Im}\mathbb{H}\subset\mathbb{H}$ as the subset of pure imaginary 
quaternions with length $1/\sqrt{2}$. Under these considerations, the space
$E_-$ is spanned by the vectors
$$
\left\{\frac{1\wedge i-j\wedge k}{2},\,\,\frac{1\wedge j+i\wedge k}{2},\,\,
\frac{1\wedge k-i\wedge j}{2}\right\}
$$
and $E_+$ by the vectors
$$
\left\{\frac{1\wedge i+j\wedge k}{2},\,\,\frac{1\wedge j-i\wedge k}{2},\,\,
\frac{1\wedge k+i\wedge j}{2}\right\}.
$$
By straightforward elementary computations, we obtain the following lemma.
\begin{lemma}\label{planevw}
Let $x\in V$. Then the following
facts hold:
\begin{enumerate}[\rm (a)]
\item If $g(x)\neq-x$, then the $2$-plane $G(x)=x\oplus g(x)$ is generated by the orthonormal
vectors
$$
\xi(x)=-\frac{x+g(x)}{|x+g(x)|}\in{\rm Im}\mathbb{H}\quad\&\quad \eta(x)=\sqrt{2}\,g(x)\cdot \xi(x)
=\frac{1-2g(x)\cdot x}{\sqrt{2}\,|x+g(x)|}\in \mathbb{H}.
$$
\item If $g(x)=-x$, then the $2$-plane $G(x)=x\oplus (-x)$ is generated by orthonormal vectors
which have one of the
following forms:
\begin{eqnarray*}
&&\xi_1(x)=\frac{x\times  i}{|x\times i|}\in{\rm Im}\mathbb{H}\,\,\,\&\,\,\, \eta_1(x)=-\sqrt{2}\,x\cdot \xi_1(x)
=\frac{-\sqrt{2}\,x\cdot (x\times i)}{|x\times i|}\in{\rm Im}\mathbb{H},\,\,\text{if}\,\,\, x_1\neq \pm 1,\\
&&\xi_2(x)=\frac{x\times j}{|x\times j|}\in{\rm Im}\mathbb{H}\,\,\,\&\,\,\, \eta_2(x)=-\sqrt{2}\,x\cdot \xi_2(x)
=\frac{-\sqrt{2}\,x\cdot(x\times j)}{|x\times j|}\in{\rm Im}\mathbb{H},\,\,\text{if}\,\, x_2\neq \pm 1,\\
&&\xi_3(x)=\frac{x\times k}{|x\times k|}\in{\rm Im}\mathbb{H}\,\,\,\&\,\,\, \eta_3(x)=-\sqrt{2}\,x\cdot \xi_3(x)
=\frac{-\sqrt{2}\,x\cdot(x\times k)}{|x\times k|}\in{\rm Im}\mathbb{H},\,\,\text{if}\,\, x_3\neq \pm 1.
\end{eqnarray*}
\end{enumerate}
\end{lemma}
Now we give the classification of Gluck and Warner \cite{gluck3} of the great circle fibrations of the $3$-sphere
following our approach.
\begin{theorem}\label{localhopf}
Suppose that $G:V\subset\S^2_-\to\S^2_-\times\S^2_+\simeq\mathbb{G}_2(\R^4)$ is the graph of
a smooth map $g:V\subset\S^2_-\to\S^2_+$, where $V$ is a path connected domain of the sphere.
The following statements hold:
\begin{enumerate}[\rm(a)]
\item Let $x_0$ be a point in $V$. Then, the circle $G(x_0)\cap\S^3$ can be represented by
$$
S(x_0,t)=\cos t\, \xi(x_0)+\sin t\,\eta(x_0),
$$
where $\xi$ and $\eta$ are the vectors obtained in Lemma {\rm \ref{planevw}}. In particular:
\medskip
\begin{enumerate}
\item[\rm (a$_1$)]
If $g(x_0)\neq \pm x_0$,
there exist an open neighbourhood $U_{x_0}\subset V$ and a positive number
$\varepsilon_{x_0}>0$, such that the map $S:U_{x_0}\times(-\varepsilon_{x_0},\varepsilon_{x_0})\to\S^3$ given by
$$
S(x,t)=\cos t\, \xi(x)+\sin t\,\eta(x)=\big(\cos t+\sqrt{2}\sin t\,g(x)\big)\cdot \xi(x),
$$
is a diffeomorphism.
\medskip
\item[\rm (a$_2$)] If $g(x)=x$ in an open set $U_{x_0}$ around $x_0$, then
$S:U_{x_0}\times(-\pi/2,\pi/2)\to\S^3$ given by
$$
S(x,t)=\cos t\, \xi(x)+\sin t\,\eta(x)=\frac{-\cos t\,x+\sqrt{2}\,\sin t}{\sqrt{2}},
$$
is a diffeomorphism. Geometrically, the map $S$ describes the projection from the poles
$\pm 1\in\S^3\subset\mathbb{H}$
onto the equator $\mathbb{S}^2=\S^3\cap{\rm Im}{\mathbb{H}}$.
\medskip
\item[\rm (a$_3$)] If $g(x)=-x$ in an open set $U$ around $x_0$ which
does not contain the point $\pm k\in \mathbb{H}$, then there exists a sufficiently small
open neighbourhood $U_{x_0}\subset U$ and a positive number $\varepsilon_{x_0}>0$
such that the map $S:U_{x_0}\times(-\varepsilon_{x_0},
\varepsilon_{x_0})\to\S^3$ given by
$$
S(x,t)=\cos t\, \xi(x)+\sin t\,\eta(x)=\frac{\cos t\,x\times k-\sqrt{2}\,\sin t\,x\cdot(x\times k)}{|x\times k|},
$$
is a diffeomorphism. In particular, the image of $S$ lies within $\S^2=\mathbb{S}^3\cap{\rm Im}\mathbb{H}$.
\end{enumerate}
\medskip
In each of the cases {\rm ($a_1$), ($a_2$) and ($a_3$)}, the map
$\pi:S\big(U_{x_0}\times(-\varepsilon_{x_0},\varepsilon_{x_0})\big)\subset\S^3\to U_{x_0}$ given
by
$$
\pi\big(\cos t\, \xi(x)+\sin t\,\eta(x)\big)=\big(\overline{\xi(x)}\cdot g(x)\cdot \xi(x)\big)=x,
$$
is a submersion with totally geodesic fibers, which are generated by the unit vector field $\zeta=dS_{x_0}(\partial_t)$.
Additionally, if $h:\S^2\to\S^2$ is a diffeomorphism, then $h\circ \pi$ gives rise to a submersion whose fibers are
generated again by $\zeta$. 
\medskip
\item
The great circle foliation
$$\mathcal{F}=\cup_{x\in U_{x_0}}G(x)\cap\S^3$$
is smooth if and only if $g:U_{x_0}\subset\S^2_-\to\S^2_+$ is strictly length decreasing.
\medskip
\item
If $g:\S^2_-\to\S^2_+$
is strictly length decreasing, then the maps $\pi$ given in {\rm (a)} generate a globally defined submersion
from $\S^3$ onto $\S^2_-$ with totally
geodesic fibers. Moreover, the quotient map $f$ is a Hopf fibration, and $\zeta$ is a Hopf vector field, if and
only if the map $g$ is constant.
\end{enumerate}
\end{theorem}
The proof of Theorem \ref{localhopf} follows by direct computations and using Lemmata
\ref{Eq:Lemmaidentitiesquaternion} and \ref{planevw}.

\subsection{Applications} Let us collect some immediate applications of the classification theorem of the great circle fibrations of the $3$-sphere.

\begin{corollary}[Gluck \cite{gluck0}]
Let $\zeta$ be a divergence free unit vector field with totally geodesic integral curves.
The dual $1$-form associated with $\zeta$, gives rise to a contact structure of $\S^3$.
\end{corollary}

\begin{proof} Let us denote with $\omega$
the associated $1$-form to $\zeta$, i.e. the form given by $\omega(\alpha)=\langle\zeta,\alpha\rangle$, for all tangent vectors
$\alpha$. Recall that $\omega$ is a contact form in $\S^3$ if and only if
$\omega\wedge d\omega\neq 0.$
Let now $\{\alpha_1=\zeta,\alpha_2,\alpha_3\}$ be a local orthonormal frame and $\varphi$ the tensor introduced in
the second section. We compute
\begin{eqnarray*}
(\omega\wedge d\omega)(\alpha_1,\alpha_2,\alpha_3)&=&d\omega(\alpha_2,\alpha_3)
=\alpha_2(\omega(\alpha_3))-\alpha_3(\omega(\alpha_2))-\omega([\alpha_2,\alpha_3])\\
&=&\langle\nabla_{\alpha_2}\alpha_1,\alpha_3\rangle-\langle\nabla_{\alpha_3}\alpha_1,\alpha_2\rangle
=-\varphi_{23}+\varphi_{32}\\
&=&\trace(\varphi\circ J),
\end{eqnarray*}
where $J$ stands for the complex structure of $\mathcal{H}$.
Recall that in the proof of Lemma \ref{gw}(c) we proved that
$\trace(\varphi\circ J)$ is nowhere zero. Hence, $\omega$ is a contact form. This completes the proof.
\end{proof}

\begin{corollary}[Gluck \& Gu \cite{gluck1}]\label{gluckgu}
Let $\zeta$ be a divergence free unit vector field with totally geodesic integral curves.
Then, $\zeta$ is a Hopf vector field.
\end{corollary}

\begin{proof}
Since $\zeta$ is divergence free, then $\trace\varphi=0$. Then, from Lemma \ref{ricattieqns}(b) it follows that $\det\varphi=1$.
Hence, from Lemma \ref{gw}(c) it follows that the singular values of $g$ are equal. Consequently, $g$ is a conformal map. This
implies that the graph of $g$ is a minimal surface in $\S^2_-\times\S^2_+$; see for example \cite[Proposition 4.5.3]{baird1},
\cite{eells} or \cite{savas-4}.
On the other hand, $g$ is strictly length decreasing. From \cite[Theorem A]{savas1} it follows that $g$ must be constant.
Consequently, from Theorem \ref{localhopf}(c) we deduce that $\pi$ is a Hopf fibration and so $\zeta$ is a Hopf vector
field. This completes the proof.
\end{proof}

\begin{corollary}[Heller \cite{heller}]\label{heller}
Let $f:\S^3\to\S^2$ be a weakly conformal submersion with totally geodesic fibers. Then the map
$f$ is the composition of a Hopf fibration with a conformal diffeomorphism of $\S^2$.
\end{corollary}

\begin{proof}
Choose a local orthonormal frame $\{\alpha_1,\alpha_2,\alpha_3\}$, such that $\alpha_1\in\ker df$.
Denote also by $\lambda^2$ the conformal factor of $f$, i.e.
\begin{equation}\label{confphi}
\lambda^2=\langle df(\alpha_2),df(\alpha_2)\rangle=\langle df(\alpha_3),df(\alpha_3)\rangle>0.
\end{equation}
In this case it holds $|df|^2=2\lambda^2>0$ and so $\lambda$ is a smooth function. Differentiating the identity
of \eqref{confphi} with respect to $\alpha_1$, we deduce that
\begin{eqnarray*}
\lambda\alpha_1(\lambda)&=&\langle \nabla_{\alpha_1}df(\alpha_2),df(\alpha_2)\rangle=
\langle B(\alpha_1,\alpha_2)+df(\nabla_{\alpha_1}\alpha_2),df(\alpha_2)\rangle\\
&=&\langle B(\alpha_1,\alpha_2),df(\alpha_2)\rangle=\langle \nabla_{\alpha_2}df(\alpha_1)-df(\nabla_{\alpha_2}\alpha_1),df(\alpha_2)\rangle\\
&=&\lambda^2\varphi_{22},
\end{eqnarray*}
and
\begin{eqnarray*}
\lambda\alpha_1(\lambda)&=&\langle \nabla_{\alpha_1}df(\alpha_3),df(\alpha_3)\rangle=
\langle B(\alpha_1,\alpha_3)+df(\nabla_{\alpha_1}\alpha_3),df(\alpha_3)\rangle\\
&=&\langle B(\alpha_1,\alpha_3),df(\alpha_3)\rangle=\langle \nabla_{\alpha_3}df(\alpha_1)-df(\nabla_{\alpha_3}\alpha_1),df(\alpha_3)\rangle\\
&=&\lambda^2\varphi_{33}.
\end{eqnarray*}
Moreover,
\begin{eqnarray*}
0&=&\alpha_1\langle df(\alpha_2),df(\alpha_3)\rangle\\
&=&
\langle B(\alpha_1,\alpha_2)+df(\nabla_{\alpha_1}\alpha_2),df(\alpha_3)\rangle
+\langle B(\alpha_1,\alpha_3)+df(\nabla_{\alpha_1}\alpha_3),df(\alpha_2)\rangle\\
&=&\langle B(\alpha_1,\alpha_2),df(\alpha_3)\rangle+\langle B(\alpha_1,\alpha_3),df(\alpha_2)\rangle
+\lambda^2\langle \nabla_{\alpha_1}\alpha_2,\alpha_3\rangle
+\lambda^2\langle \nabla_{\alpha_1}\alpha_3,\alpha_2\rangle\\
&=&-\langle df(\nabla_{\alpha_2}\alpha_1),df(\alpha_3)\rangle-\langle df(\nabla_{\alpha_3}\alpha_1),df(\alpha_2)\rangle\\
&=&\lambda^2(\varphi_{23}+\varphi_{32}).
\end{eqnarray*}
From the above relations we see that
\begin{equation}\label{hellerequalities}
\varphi_{23}=-\varphi_{32}\quad\text{and}\quad\varphi_{22}=\varphi_{33}=\alpha_1(\log\lambda).
\end{equation}
As a consequence, $|\varphi|^2=2\det\varphi.$
Hence, from Lemma \ref{gw}(c)
we see that the singular values of $g$ are everywhere equal. As in the proof of Corollary \ref{gluckgu}, we show that $g$ is constant.
This implies that $\alpha_1$ is a Hopf vector field and the corresponding projection $\pi$
a Hopf fibration. Observe that there exist a diffeomorphism $h:\S^2\to\S^2$
such that $f=h\circ\pi$. Since $f$ is a weakly conformal submersion and $\pi$ a Riemannian
submersion, we deduce that $h$ is a conformal diffeomorphism. This completes the proof.
\end{proof}

\begin{remark}
In \cite{heller} Heller proves a more general result. More precisely, he shows that
{\em up to conformal transformations of $\S^2$ and $\S^3$, every conformal fibration
of $\S^3$ by circles $($not necessarily great circles$)$ is the Hopf fibration.}
\end{remark}

\begin{corollary}[Escobales \cite{escobales}]
Suppose that $f:V\subset\S^3\to\S^2$ is a submersion with totally geodesic fibers and equal constant singular values
defined in an open neighbourhood $V$ of $\S^3$. Then $f$ is a Hopf fibration.
\end{corollary}
\begin{proof}
By assumption, the singular values of $f$ are $0,\lambda,\lambda$, where $\lambda$ is a non-zero constant.
From the formulas \eqref{hellerequalities}, we deduce that
$\varphi_{22}=\varphi_{33}=0$ and $\varphi_{23}+\varphi_{32}=0.$
From Lemma \ref{ricattieqns}(b) it follows that $\det\varphi=1$, which implies that
$\varphi_{23}=-1$ and $\varphi_{32}=1.$
Now
from Lemma \ref{gw}(c) the singular values of $g$ are zero, which implies that
$g$ is constant. Following the same lines as in Corollary \ref{heller},  we deduce that
$f$ is a Hopf fibration. This completes the proof.
\end{proof}

\section{Harmonic and minimal unit vector fields}\label{sec3}

Let us assume now that $g:V\subset\S^2_-\to\S^2_+$ is strictly length decreasing, where $V$ is an open saturated set. Fix a point $x_0\in\S^2_-$.
Then from Theorem \ref{localhopf}, it follows that there exists a sufficiently small neighbourhood $U_{x_0}$
around the point $x_0$ where the map $\xi$ is a diffeomorphism.  Hence, by setting $y=\xi(x)$
and $h=\sqrt{2}\,g\circ \xi^{-1}$,
we may reparametrise the foliation by $\vartheta:\xi(U_{x_0})\times[-\pi,\pi]\to\S^3$
given by
$$
\vartheta(y,t)=\cos t\, y+\sin t\, h(y)\cdot y.
$$
Observe that $h$ a smooth map from an open domain of the unit $2$-sphere
$\S^3\cap{\rm Im}\mathbb{H}$ with values in the same sphere. According
to Theorem \ref{localhopf}, we have that
$$
\zeta(\vartheta(y,t))=h(y)\cdot \vartheta(y,t)\quad\text{and}\quad \pi(\vartheta(y,t))=\frac{\overline{y}\cdot h(y)\cdot y}{\sqrt{2}}.
$$
On the other hand, one can readily check that for any $(y,t)$ it holds
\begin{equation}\label{pi}
\pi(\vartheta(y,t))=\frac{\overline{\vartheta(y,t)}\cdot h(y)\cdot\vartheta(y,t)}{\sqrt{2}},
\end{equation}
Hence, for $p=\vartheta(y,t)$ we deduce the unit vector field $\zeta$ generating the leaves of the foliation is related with $\pi$ via the
equations
\begin{equation}\label{Eq: relationfe1}
\zeta(p)=p\cdot f(p)\quad\text{and}\quad f(p)=\overline{p}\cdot \zeta(p),
\end{equation}
for any $p \in V\subset\mathbb{S}^{3}$, where $f=\sqrt{2}\,\pi:\S^3\to\S^2$.

\subsection{Harmonic unit vector fields} We would like to relate the Hessian of the vector field $\zeta$ with the Hessian of the quotient map $f$
given in \ref{Eq: relationfe1}.
We need the following auxiliary lemma.

\begin{lemma}
Let $\zeta$ be a unit vector field with totally geodesic integral curves defined in an open
saturated  neighbourhood
$V$ of $\S^3$ and $f:V\subset\S^3\to\S^2$ the corresponding quotient map given in {\rm\ref{Eq: relationfe1}}. Then,
\begin{equation}\label{Eq: df}
\langle df(\alpha_i), df(\alpha_j)\rangle=\langle \alpha_i, \alpha_j\rangle
+\langle \varphi(\alpha_i), \varphi(\alpha_j)\rangle-\langle \varphi(\alpha_i), J\alpha_j \rangle
-\langle\varphi(\alpha_j), J\alpha_i\rangle,
\end{equation}
where $i,j\in\{2,3\}$, $\alpha_2, \alpha_3 \in \mathcal{H}$ and $J$ is the complex structure of $\mathcal{H}$.
\end{lemma}
\begin{proof}
Observe at first that any tangent vector $\alpha \in T_{p}\S^{3}$ satisfies the equations
\begin{equation}\label{Eq: bar X}
\overline{\alpha}=-\overline{p}\cdot \alpha \cdot \overline{p}\quad\text{and}\quad \alpha=-p\cdot \overline{\alpha} \cdot p.
\end{equation}
Indeed! Consider a curve $\sigma:(-\varepsilon, \varepsilon) \to \S^{3}$ such that $\sigma(0)=p$ and
$\sigma'(0)=\alpha$. Differentiating the expression $\sigma \cdot \overline{\sigma}=1$, and estimating
at $t=0$, we get the first identity of \eqref{Eq: bar X}. The second follows immediately from the first.
Differentiating the second equation of (\ref{Eq: relationfe1}) with respect to $\alpha_i$, $i,j\in\{2,3\}$, we get
\begin{equation}\label{Eq: differential f}
df(\alpha_i)=\overline{\alpha_i}\cdot \zeta+\overline{p}\cdot D_{\alpha_i}\zeta=
\overline{\alpha_i}\cdot \zeta+\overline{p}\cdot \nabla_{\alpha_i}\zeta
=\overline{\alpha_i}\cdot \zeta-\overline{p}\cdot \varphi(\alpha_i),
\end{equation}
where $D$ is the standard connection of $\mathbb{H}=\R^4$. Using  Lemma \ref{Eq:Lemmaidentitiesquaternion} and (\ref{Eq: bar X}), we obtain
\begin{eqnarray}\label{Eq: product df}
\langle df(\alpha_i), df(\alpha_j)\rangle&=&\langle \overline{\alpha_i}\cdot \zeta-\overline{p}\cdot \varphi(\alpha_i), \overline{\alpha_j}\cdot \zeta-\overline{p}\cdot \varphi(\alpha_j)\rangle\\
&=&\langle \alpha_i, \alpha_j \rangle+\langle\varphi(\alpha_i), \varphi(\alpha_j)\rangle
+\langle \alpha_i\cdot \overline{p}\cdot \zeta, \varphi(\alpha_j)\rangle
+\langle \alpha_j\cdot \overline{p}\cdot \zeta, \varphi(\alpha_i) \rangle.\nonumber
\end{eqnarray}
For any $p\in\S^3$ the vectors $\{\alpha_{1}=\zeta, \alpha_{2}, \alpha_{3},p\}$ forms a basis of $\mathbb{H}$. We decompose the vector $\alpha_{2}\cdot\overline{p}\cdot\alpha_{3}$ in terms of $\alpha_{1}, \alpha_{2}, \alpha_{3}$ and $p$. We easily observe that
$$\langle \alpha_{2}\cdot\overline{p}\cdot\alpha_{3}, \alpha_{2}\rangle=\langle\alpha_{2}\cdot\overline{p}\cdot\alpha_{3}, \alpha_{3}\rangle=0.$$
Furthermore,
\begin{equation}
\langle \alpha_{2}\cdot\overline{p}\cdot\alpha_{3}, p\rangle=\langle \alpha_{2}\cdot\overline{p}\cdot\alpha_{3}\cdot\overline{p}, 1\rangle=-\langle\alpha_{2}\cdot\overline{\alpha_{3}}, 1\rangle=0,
\end{equation}
since $\langle \alpha_{2}, \alpha_{3}\rangle=0={\rm Re}(\alpha_{2}\cdot\overline{\alpha_{3}})$. Hence, $\alpha_{2}\cdot\overline{p}\cdot\alpha_{3}=\pm \alpha_{1}$. Choosing positive orientation, we get
\begin{equation}\label{Eq: product p alpha I}
\begin{array}{lll}
\alpha_{1}\cdot\overline{p}\cdot\alpha_{2}=-\alpha_{2}\cdot\overline{p}\cdot\alpha_{1}=\alpha_{3},\\ \alpha_{3}\cdot\overline{p}\cdot\alpha_{1}=-\alpha_{1}\cdot\overline{p}\cdot\alpha_{3}=\alpha_{2},\\
\alpha_{2}\cdot\overline{p}\cdot\alpha_{3}=-\alpha_{3}\cdot\overline{p}\cdot\alpha_{2}=\alpha_{1}.
\end{array}
\end{equation}
Using (\ref{Eq: bar X}), we obtain
\begin{equation}\label{Eq: product p alpha II}
\begin{array}{ll}
\alpha_{1}\cdot\overline{p}\cdot \alpha_{1}=-\alpha_{1}\cdot\overline{\alpha_{1}}\cdot p=-p, & \alpha_{2}\cdot\overline{p}\cdot \alpha_{2}=\alpha_{3}\cdot\overline{p}\cdot \alpha_{3}=-p.
\end{array}
\end{equation}
From (\ref{Eq: product df}) and the equations (\ref{Eq: product p alpha I}), we easily get (\ref{Eq: df}). This completes
the proof.
\end{proof}

\begin{definition}
Let $F:(M,g_M,\nabla^M)\to(N,g_N,\nabla^N)$ be a smooth map between Riemannian manifolds and $\nabla^F$ the induced by $F$
connection of the pullback bundle $F^*TN$. The tensor $B_F$ given by
$$
B_F(X,Y)=\nabla^F_{Y}dF(X)-dF(\nabla^M_XY), \quad X,Y\in \mathfrak{X}(M),
$$
is called the Hessian of $F$ and its trace
$$
\tau_F={\trace}_{g_M}B_F
$$
is called the tension field of $F$. The map $F$ is called harmonic if $\tau_F=0$.
\end{definition}

\begin{proposition}\label{lapetensf}
Let $\zeta$ be a unit vector field with totally geodesic integral curves defined in an open
saturated  neighbourhood
$V$ of $\S^3$ and $f:V\subset\S^3\to\S^2$ the corresponding quotient map given in {\rm \ref{Eq: relationfe1}}.  Then the Hessian of $\zeta$, at a point $p\in\S^3$, satisfies the formula
\begin{equation}\label{Eq: Laplaciane1}
\nabla^2_{\alpha_i, \alpha_j}{\zeta}=\nabla_{\alpha_i}\nabla_{\alpha_j}\zeta-\nabla_{\nabla_{\alpha_i}\alpha_j}\zeta=p \cdot B_{f}(\alpha_i, \alpha_j)-\langle\nabla_{\alpha_i}\zeta,\nabla_{\alpha_j}\zeta\rangle \zeta,
\end{equation}
where $i,j\in\{2,3\}$, $\{\alpha_2, \alpha_3\}$ is a local orthonormal frame of $\mathcal{H}$, and $B_f$ the Hessian of the map $f:(\S^3,g_{\S^3})\to(\S^2,g_{\S^2})$. In particular,
$$
\Delta \zeta+|\nabla \zeta|^2\zeta=p\cdot \tau_{f},
$$
where here $\Delta$ stands for the rough Laplacian operator and $\tau_{f}$ for the tension field of $f$.
\end{proposition}
\begin{proof}
Consider a local orthonormal frame $\{\alpha_1=\zeta,\alpha_2,\alpha_3\}$
such that $\alpha_1\in\mathcal{V}$ and $\alpha_2,\alpha_3\in\mathcal{H}$. Differentiating the identity
$$\alpha_1(p)=p\cdot f(p)$$
with respect to $\alpha_{j}$, $j\in\{2,3\}$, we get
\begin{equation}\label{Eq: differentiation1}
D_{\alpha_j}\alpha_{1}=\alpha_j\cdot f+p\cdot df(\alpha_j).
\end{equation}
Differentiating (\ref{Eq: differentiation1}) with respect to $\alpha_i$, $i\in\{2,3\}$, we have
\begin{eqnarray}\label{Eq: differentiation2}
D_{\alpha_i}D_{\alpha_j}\alpha_{1}&=&\nabla_{\alpha_i}\alpha_j\cdot f-\langle \alpha_i, \alpha_j\rangle \alpha_{1}
-\langle df(\alpha_i), d f(\alpha_j)\rangle \alpha_{1}\nonumber\\
&&+\alpha_j\cdot d f(\alpha_i)+\alpha_i\cdot d f(\alpha_j)+p\cdot \nabla^{f}_{\alpha_i} d f(\alpha_j).
\end{eqnarray}
Using (\ref{Eq: differentiation1}) and (\ref{Eq: differentiation2}), we deduce
\begin{eqnarray}\label{Eq: differentiation3}
\nabla^{2}_{\alpha_i, \alpha_j}\alpha_{1}&=&p\cdot B_{f}(\alpha_i, \alpha_j)
-\big(\langle \alpha_i, \alpha_j\rangle+\langle d f(\alpha_i), d f(\alpha_j)\rangle\big)\alpha_{1}\nonumber\\
&&+\alpha_i\cdot df(\alpha_j)+\alpha_j\cdot df(\alpha_i)-\big(\langle \alpha_i, \varphi(\alpha_j)\rangle
+\langle \varphi(\alpha_i), \alpha_j\rangle\big)p.
\end{eqnarray}
From the formulas (\ref{Eq: product p alpha I})-(\ref{Eq: product p alpha II}), we get
\begin{eqnarray}\label{Eq: differentiation4}
\alpha_i\cdot\overline{p}\cdot \varphi(\alpha_j)&=&-\langle \varphi(\alpha_j), \alpha_i\rangle p+\langle \varphi(\alpha_j), J\alpha_i\rangle \alpha_{1}, \nonumber\\
\alpha_j\cdot\overline{p}\cdot \varphi(\alpha_i)&=&-\langle \varphi(\alpha_i), \alpha_j\rangle p+\langle \varphi(\alpha_i), J\alpha_j\rangle \alpha_{1}.
\end{eqnarray}
Combining Lemma \ref{Eq:Lemmaidentitiesquaternion}, (\ref{Eq: bar X}), (\ref{Eq: differential f}) and (\ref{Eq: differentiation4}), we obtain
\begin{eqnarray}\label{Eq: differentiation5} 
\alpha_i\cdot df(\alpha_j)+\alpha_j\cdot df(\alpha_i)&=&-\big(\alpha_i\cdot \overline{p} \cdot \alpha_j \cdot \overline{p} +\alpha_j\cdot \overline{p} \cdot \alpha_i \cdot \overline{p}\big)\alpha_{1}\nonumber\\
&&-\alpha_i\cdot\overline{p}\cdot \varphi(\alpha_j)-\alpha_j\cdot\overline{p}\cdot \varphi(\alpha_i)\nonumber\\
&=&2\langle \alpha_i, \alpha_j \rangle \alpha_{1}+\big(\langle \varphi(\alpha_j), \alpha_i \rangle
+\langle \varphi(\alpha_i), \alpha_j\rangle\big)p\nonumber\\
&&-\big(\langle \varphi(\alpha_j), J\alpha_i\rangle+\langle \varphi(\alpha_i), J\alpha_j\rangle \big)\alpha_{1}.
\end{eqnarray}
Substituting (\ref{Eq: differentiation5}) in (\ref{Eq: differentiation3}) and using (\ref{Eq: df}), we obtain
the desired formula (\ref{Eq: Laplaciane1}). This completes the proof of the proposition.
\end{proof}

As a direct consequence of Proposition \ref{lapetensf} we derive our first main result of the paper.

\begin{thmb}
Let $\zeta$ be a unit vector field with totally geodesic integral curves defined in an open
saturated  neighbourhood
$V$ of $\S^3$ and $f:V\subset\S^3\to\S^2$ the corresponding quotient map. Then $f$ is a harmonic map if and
only if $\zeta:V\to U\S^3$ is a harmonic unit vector field. 
\end{thmb}

Now let us state and prove the our next main theorem.
\begin{thmc}
A harmonic unit vector field $\zeta$ on $\S^3$, whose integral curves are great circles, is a Hopf vector field and the corresponding quotient map $f:\S^3\to\S^2$ is a Hopf fibration.
\end{thmc}

\begin{proof} Consider a local orthonormal frame $\{\alpha_1=\zeta,\alpha_2,\alpha_3\}$
such that $\alpha_2,\alpha_3\in\mathcal{H}$.
Let us introduce the functions
$$v=\trace\varphi=\varphi_{22}+\varphi_{33}\quad\text{and} 
\quad
u=\trace(\varphi\circ J)=\varphi_{32}-\varphi_{23}.$$
The functions $v$ and $u$ are smooth and globally defined on $\S^3$.
Denote by $\pi:\S^3\to\S^2$ the associated quotient map to $\alpha_1$.
We have that
$$
\Delta \alpha_1+|\nabla \alpha_1|^2\alpha_1=0.
$$
From the equations
$$\langle\Delta\alpha_1,\alpha_2\rangle=0=\langle\Delta\alpha_1,\alpha_3\rangle$$
we obtain
\begin{equation}\label{Eq: condition 1 harmonicity}
\alpha_{2}(\varphi_{22})+\alpha_{3}(\varphi_{32})=(\varphi_{22}-\varphi_{33})\langle \nabla_{\alpha_{3}}\alpha_{3}, \alpha_{2}\rangle+(\varphi_{23}+
\varphi_{32})\langle \nabla_{\alpha_{2}}\alpha_{2}, \alpha_{3}\rangle
\end{equation}
and
\begin{equation}\label{Eq: condition 2 harmonicity}
\,\alpha_{2}(\varphi_{23})+\alpha_{3}(\varphi_{33})=(\varphi_{33}-\varphi_{22})\langle \nabla_{\alpha_{2}}\alpha_{2}, \alpha_{3}\rangle+(\varphi_{23}+
\varphi_{32})\langle \nabla_{\alpha_{3}}\alpha_{3}, \alpha_{2}\rangle.
\end{equation}
Since
$$(\nabla^{\mathcal{H}}_{\alpha_{2}}\varphi)\alpha_{3}=(\nabla^{\mathcal{H}}_{\alpha_{3}}\varphi)\alpha_{2},$$
we obtain the following system of PDEs
\begin{equation}\label{Eq: condition 1 Codazzi}
\alpha_{3}(\varphi_{22})-\alpha_{2}(\varphi_{32})=(\varphi_{22}-\varphi_{33})\langle \nabla_{\alpha_{2}}\alpha_{2}, \alpha_{3}\rangle-(\varphi_{23}+
\varphi_{32})\langle \nabla_{\alpha_{3}}\alpha_{3}, \alpha_{2}\rangle
\end{equation}
and
\begin{equation}\label{Eq: condition 2 Codazzi}
\,\alpha_{3}(\varphi_{23})-\alpha_{2}(\varphi_{33})=(\varphi_{23}+\varphi_{32})\langle \nabla_{\alpha_{2}}\alpha_{2}, \alpha_{3}\rangle+(\varphi_{22}-\varphi_{33})\langle \nabla_{\alpha_{3}}\alpha_{3}, \alpha_{2}\rangle.
\end{equation}
Subtracting (\ref{Eq: condition 2 Codazzi}) from (\ref{Eq: condition 1 harmonicity}), adding
(\ref{Eq: condition 2 harmonicity}) and (\ref{Eq: condition 1 Codazzi}) and keeping in mind \eqref{ric2},
we deduce that
the functions $u$ and $v$ satisfy the system of differential equations
\begin{equation}\label{cauchyriemann}
\alpha_1(v)=v^2-2(1+\det\varphi)+4,\,\, \alpha_1(u)=uv,\,\,\alpha_2(u)=\alpha_3(v)\,\,\text{and}\,\,\alpha_3(u)=-\alpha_2(v).
\end{equation}
From \eqref{cauchyriemann} and (\ref{ric2}), we deduce that
\begin{eqnarray}
\Delta v&=&\alpha_{1}\alpha_{1}(v)+\alpha_{2}\alpha_{2}(v)+\alpha_{3}\alpha_{3}(v)-\nabla_{\alpha_{2}}\alpha_{2}(v)-\nabla_{\alpha_{3}}\alpha_{3}(v)\nonumber\\
&=&2v \alpha_{1}(v)-2v(1+\det\varphi)-[\alpha_{2}, \alpha_{3}](u)-v\alpha_{1}(v)\nonumber\\
&&+\langle \alpha_{2}, \nabla_{\alpha_{2}}\alpha_{3} \rangle \alpha_{2}(u)-\langle \alpha_{3}, \nabla_{\alpha_{3}}\alpha_{2}\rangle\alpha_{3}(u)\nonumber\\
&=&u^2 v+v \alpha_{1}(v)-2v(1+\det\varphi)=v(u^2+v^2-4\det\varphi)\nonumber\\
&=&v\big(|\varphi|^2-2\det\varphi\big).\nonumber
\end{eqnarray}
Since
$|\varphi|^2-2\det\varphi\ge 0,$
we obtain that
$$
\Delta v^2=2v\Delta v+2|\nabla v|^2=2v^2(|\varphi|^2-2\det\varphi)+2|\nabla v|^2\ge 0.
$$
From the maximum principle it follows that the function $v$ is constant. On the other hand
$$
v=\varphi_{22}+\varphi_{33}=-\operatorname{div}(\alpha_1)
$$
and by the Stokes' Theorem we conclude that $v=\trace\varphi=0$. Moreover, from
the first equation of \eqref{cauchyriemann} we deduce that $\det\varphi=1$. From Lemma
\ref{gw}(c) it follows now that the associated graph in $\S^2_-\times\S^2_+$ is generated
by a conformal and
strictly length decreasing map $g$. Note that conformality implies minimality; see \cite{eells}.
Then, from \cite[Theorem A]{savas1},
it follows that $g$ must be constant. From Theorem \ref{localhopf}(c) we
deduce that $\alpha_1$ is a Hopf vector field. 
\end{proof}

\subsection{Minimal unit vector fields}\label{sec3}
The unit tangent sphere bundle of $\S^3$ can be identified
with the {\em Stiefel manifold} $V_{2}(\R^4)$, i.e. the set of orthonormal two-frames in $\R^4$.
The Stiefel manifold $V_{2}(\R^{4})$ sits as a circle bundle over the Grassmann manifold
$\mathbb{G}_{2}(\R^{4})$ of oriented two-planes in $\R^4$. 
On the other hand, using the quaternionic structure of $\R^4$, we can identify $U\S^3$ with the product of unit spheres $\S^3 \times \S^2$. Indeed, one can easily check that the map
$\varPhi: U\S^3\to \S^3 \times \S^2$ given by
\begin{eqnarray*}
\varPhi(x, v)=(x, x\cdot v)
\end{eqnarray*}
is a diffeomorphism. The map $T:V_{2}(\R^{4})\cong U\S^3\to \mathbb{G}_{2}(\R^{4})$
given by
\begin{eqnarray*}
T(x, u)=x\wedge u.
\end{eqnarray*}
is called the {\em Stiefel bundle map}. It turns out that
$T:(U\S^3, g_{S}) \to (\mathbb{G}_{2}(\R^{4}), (\cdot\,,\cdot))$ is a Riemannian submersion; see for
example \cite[pages 121-123 and Remark 7.4]{gluck3}. For reader's convenience we include a short proof of this fact.

\begin{proposition} The Stiefel bundle map $T$ is a Riemannian submersion whose fibers are generated
by the geodesic flow vector field $\xi$.
\end{proposition}
\begin{proof}
Let $(x, v)\in U\S^3$ and $Z \in T_{(x, v)}U\S^3$ and a smooth curve $\gamma: (-\varepsilon, \varepsilon) \to U\S^3$ such that $\gamma(0)=(x, v)$ and $\gamma'(0)=Z$. Then, $\gamma$ can be written in the form
$\gamma=(\alpha, V),$
where $\alpha=\pi \circ \gamma$  is  the projection of the curve $\gamma$ to $\S^3$ and $V$ is a vector field along $\alpha$. Note that $\alpha(0)=x$ and $V(0) =v$. Furthermore, we have that
$$Z=w_{1}^{hor}+w_{2}^{tan},$$
where $w_{1}=\alpha'(0)$ and $w_{2}=V'(0)$. Differentiating $T\circ\gamma=\alpha\wedge V$ with
respect to $t$ and estimating at $t=0$, we get
\begin{equation}\label{Eq: differential of T}
dT_{(x, v)}(Z)=dT_{(x, v)}(w_{1}^{hor}+w_{2}^{tan})=w_{1}\wedge v+x \wedge w_{2}.
\end{equation}
Fix a point $ x \in \S^3$ and consider a local orthonormal frame $\{v, \alpha_{2}, \alpha_{3}\}$ of $T_{x}\S^3$. Observe that $Z \in \ker dT_{(x, v)}$ if and only if
$$w_{1}\wedge v+x \wedge w_{2}=0.$$
Decomposing the vectors $w_{1}$ and $w_{2}$ in terms of $v, \alpha_{2}$ and  $\alpha_{3}$, we see that
$w_{1}\in \span{v}$ and $w_{2}=0$. Clearly the Stiefel bundle map is a submersion and the geodesic flow vector field $\xi$ generates the fibers of $T$. Furthermore, $Z \in \{\ker dT_{(x, v)}\}^{\perp}$ if and only if
\begin{equation}\label{Eq: T 1}
\langle w_{1}, v\rangle=0.
\end{equation}
Differentiating the expression $\langle V, V\rangle =1$ with respect to $t$ and evaluating at $t=0$, we easily obtain that
\begin{equation}\label{Eq: T 2}
\langle w_{2}, v\rangle=0.
\end{equation}
It remains to show that the map $T$ is a Riemannian submersion. Indeed, choose vector fields
$$
Z_{1}=w_{1}^{hor}+w_{2}^{tan}\in\{\ker dT_{(x, v)}\}^{\perp}\quad\text{and}\quad Z_{2}=w_{3}^{hor}+w_{4}^{tan} \in \{\ker dT_{(x, v)}\}^{\perp}.
$$
Using (\ref{Eq: T 1}) and (\ref{Eq: T 2}), we have
\begin{eqnarray}
(d T_{(x, v)}(Z_{1}), d T_{(x, v)}(Z_{2}))&=&(w_{1}\wedge v+x \wedge w_{2}, w_{3}\wedge v+x \wedge w_{4})\nonumber\\
&=&\langle w_{1}, w_{3}\rangle +\langle w_{2}, w_{4}\rangle -\langle w_{2}, x\rangle \langle w_{4}, x\rangle\nonumber\\
&=&g_{S}(Z_{1}, Z_{2}).\nonumber
\end{eqnarray}
This completes the proof of the proposition.
\end{proof}

We would like to relate the mean curvature of the vector field $\zeta$ with the mean curvature of the graph of the strictly 
length decreasing map $g:V\subset\S^2_-\to\S^2_+$. We need the following lemma.

\begin{lemma}\label{minimalzminimalg}
Let $\zeta:V\subset\S^3\to U\S^3$ be a unit vector field with totally geodesic integral curves defined in an saturated
neighbourhood $V$ of $\S^3$, let $\pi:V\subset\S^3\to\S^2_-$ the quotient map given in {\rm \ref{pi}} and let
$G:\pi(V)\subset\S_{-}^2\to\S_{-}^2\times\S_{+}^2$ be its corresponding graph.
\begin{enumerate}[\rm(a)]
\item
The quotient map
$\pi:(V\subset\S^3,\zeta^*g_S)\to (\S_{-}^2,G^*g_{\S^2_-\times\S^2_+})$
is a Riemannian submersion. Moreover, we have that
$T\circ\zeta=G\circ \pi$; compare also with \cite[equation (7.12) page 126]{gluck3}.
\medskip
\item
The following formula holds
$$
dT\big(A_{\zeta}(X, Y)\big)=A_G\big(d\pi(X), d\pi(Y)\big)+dG\big(B_{\pi}(X, Y)\big)
$$
for all vector fields $X, Y$ on $(V, \zeta^{*}g_{S})$ perpendicular to $\zeta$, where $T$ is the Stiefel map and $B_{\pi}$
the Hessian of $\pi:(V,\zeta^*g_S)\to (\S_{-}^2,G^*g_{\S^2_-\times\S^2_+})$. In particular,
\begin{equation}\label{Eq: relation mean curvatures}
dT\big(H_\zeta\big)=H_G+dG(\tau_{\pi}),
\end{equation}
where $\tau_{\pi}$ is the tension field of $\pi:(V,\zeta^*g_S)\to (\S_{-}^2,G^*g_{\S^2_-\times\S^2_+})$.
\end{enumerate}
\end{lemma}
\begin{proof}
(a) From the definition of the Stiefel map, we have
\begin{equation*}
T(\zeta(x))=T(x, \zeta_x)=x\wedge \zeta_x=(g_{-}(\pi(x)), g_{+}(\pi(x)))=G(\pi(x)),
\end{equation*}
for all $p\in V$. Hence, $T\circ \zeta=G\circ \pi$. Since the unit vector field $\zeta$ has
geodesic integral curves, making use of (\ref{Eq: pull back metric Sasaki}), we get
\begin{eqnarray}\label{Eq: equivalence}
\{\ker dT_{\zeta}\}^{\perp}&=&\big\{W=X^{hor}+Y^{tan} \in T_{\zeta}U\S^3: \langle X, \zeta\rangle=0\big\}\nonumber\\
&=&\big\{W=X^{hor}+Y^{tan} \in T_{\zeta}U\S^3: \zeta^{*}g_{S}(X, \zeta)=0\big\}.
\end{eqnarray}
We claim now that $\pi:(V, \zeta^{*}g_{S})\to (\S^2_-, G^*g_{\S^2_-\times\S^2_+})$
is a Riemannian submersion. Indeed, from the equations (\ref{Eq: differential of zeta}), (\ref{Eq: pull back metric Sasaki}),
(\ref{Eq: equivalence}) and the fact that $T$ is a Riemannian submersion, we obtain
\begin{eqnarray}
G^*g_{\S^2_-\times\S^2_+}(d\pi(X), d\pi(Y))&=&g_{\S_{-}^2\times\S_{+}^2}\big(dG(d\pi(X)), dG(d\pi(Y))\big)\nonumber\\
&=&g_{\S_{-}^2\times\S_{+}^2}\big(dT_{\zeta}(d\zeta(X)), dT_{\zeta}(d\zeta(Y))\big)\nonumber\\
&=&g_{S}\big(X^{hor}+(\nabla_{X}\zeta)^{tan}, Y^{hor}+(\nabla_{Y}\zeta)^{tan}\big)\nonumber\\
&=&\zeta^{*}g_{S}(X, Y),\nonumber
\end{eqnarray}
for all vector fields $X, Y$ on $V\subset\S^3$ perpendicular to $\zeta$.

(b) Using the Koszul formula and the fact that $T$ is a Riemannian submersion we deduce that for any pair of $v_{1}, v_{2} \in \{\ker d T\}^{\perp}$, we have
\begin{equation}\label{Eq: Hessian T}
B_{T}(v_{1}, v_{2})=\nabla^{T}_{v_{1}}dT(v_{2})-dT(\nabla^{g_{S}}_{v_{1}}v_{2})=0,
\end{equation}
for details see also \cite[Lemma 4.5.1, page 119]{baird1}.
Since $T\circ \zeta=G\circ \pi$, we have
$$B_{T\circ \zeta}=B_{G\circ \pi}.$$
From the composition
formula, we obtain that for all vector fields $X, Y\in\mathfrak{X}(V)$, it holds
\begin{equation*}
dT(A_{\zeta}(X, Y))+B_{T}(d \zeta(X), d \zeta(Y))=d G (B_{\pi}(X, Y))+A_{G}(d\pi (X), d\pi(Y)).
\end{equation*}
Using (\ref{Eq: differential of zeta}) and (\ref{Eq: equivalence}), we easily see that if $\zeta^{*}g_{S}(X, \zeta)=0$, then $d\zeta(X)\in \{\ker d T_{\zeta}\}^{\perp}$. Therefore, (\ref{Eq: Hessian T}) gives
\begin{equation}\label{Eq: Hessian zeta graph}
dT(A_{\zeta}(X, Y))=d G (B_{\pi}(X, Y))+A_{G}(d\pi(X), d\pi(Y)),
\end{equation}
for all vector fields $X, Y$ on $(V\subset\S^3, \zeta^{*}g_{S})$ perpendicular to $\zeta$. Using the Koszul formula, we get
$$
2\zeta^{*}g_{S}(\nabla^{\zeta^{*}g_{S}}_{\zeta}\zeta, X)=-2\zeta^{*}g_{S}([\zeta, X], \zeta)=-2\langle [\zeta, X], \zeta \rangle
=2\langle \nabla_{\zeta}\zeta, X\rangle=0,
$$
for all vector fields X on $V\subset\S^{3}$ such that $\zeta^{*}g_{S}(X, \zeta)=0$. Hence, $\nabla^{\zeta^{*}g_{S}}_{\zeta}\zeta=0$. Using the identity
\begin{equation*}
\nabla^{g_{S}}_{X^{hor}}Y^{hor}=(\nabla_{X}Y)^{hor}-\tfrac{1}{2}\big(R_{\S^3}(X, Y)v\big)^{tan},
\end{equation*}
of Lemma \ref{boeckx}, we deduce that
\begin{equation}\label{Eq: Hessian zeta zeta}
A_{\zeta}(\zeta,\zeta)=\nabla^{g_{S}}_{\zeta^{hor}}\zeta^{hor}-d \zeta(\nabla^{\zeta^{*}g_{S}}_{\zeta}\zeta)=0.
\end{equation}
Taking $\trace$, with respect to the metric $\zeta^{*}g_{S}$, at the two sides of (\ref{Eq: Hessian zeta graph}), using the fact that the quotient map
$$\pi:(\S^3, \zeta^{*}g_{S}) \to (\S_{-}^2, G^*g_{\S^2_-\times\S^2_+})$$ is a Riemannian submersion and (\ref{Eq: Hessian zeta zeta}), we get (\ref{Eq: relation mean curvatures}). This completes the proof.
\end{proof}

Now we are ready to prove Theorem \ref{thmD}.

\begin{thma}
A unit vector field with totally geodesic integral curves defined in an saturated neighbourhood $V$ of $\S^3$ is minimal if and only if its
corresponding graphical surface is minimal in $\S^2\times\S^2$.
\end{thma}
\begin{proof}
We follow the notation of the above lemma. Suppose that $\zeta$ is a minimal unit vector field. Then, $H_{\zeta}=0$. Using (\ref{Eq: relation mean curvatures}), we easily obtain that $H_{G}=0$ i.e. the strictly decreasing map $g$ is a minimal map. Conversely, we assume that $g$ is a minimal map. Then,
\begin{equation}\label{Eq: minimality 1}
dT(H_{\zeta})=dG(\tau_{\pi}).
\end{equation}
Since $\xi_{\zeta}=\zeta^{hor}=d\zeta(\zeta),$
we easily observe that the mean curvature $H_{\zeta}$ is perpendicular to the geodesic flow vector field, i.e. $H_{\zeta} \in \{\ker dT_{\zeta}\}^{\perp}$. Furthermore, $dT(H_{\zeta})$ is normal to the graph of $g$. Indeed,
$$
g_{\S_{-}^2\times\S_{+}^2}\big(dT(H_{\zeta}), dG(d\pi(X))\big)=g_{\S_{-}^2\times\S_{+}^2}\big(dT(H_{\zeta}), dT(d\zeta(X))\big)
=g_{S}\big(H_{\zeta}, d\zeta(X)\big)=0,
$$
for all vector fields $X$ on $V\subset\S^3$ such that $\zeta^{*}g_{S}(X, \zeta)=0$. Using (\ref{Eq: minimality 1}), we have that $dT(H_{\zeta})=0$. Since $T$ is an isometry restricted to $\{\ker dT_{\zeta}\}^{\perp}$, we deduce that $H_{\zeta}=0$, i.e. the vector field $\zeta$ is a minimal unit vector field. This completes the proof.
\end{proof}

As an immediate consequence of  Theorem \ref{thmD} and the Bernstein type Theorem A in \cite{savas1} we obtain the following uniqueness type
theorem.

\begin{thmd}
A minimal unit vector field $\zeta$ on $\S^3$, whose integral curves are great circles, is a Hopf vector field and the corresponding quotient map $f:\S^3\to\S^2$ is a Hopf fibration.
\end{thmd}

\begin{remark}
Let us conclude the paper with some comments and final remarks.
\begin{enumerate}[\rm (a)]
\item
Let $f:\S^3\to\S^2$ be a smooth map and $\lambda_1=0\le\lambda_2\le\lambda_3$ be its singular values.
The {\em $2$-dilation} ${\rm Dil}_2(f)=\sup\lambda_2\lambda_3$ of $f$ measures
how much $f$ contract the areas of surfaces $\Sigma\subset\S^3$. In \cite{assimos}, it is shown that if ${\rm Dil}_2(f)<2$
then $f$ is null-homotopic. Since, any submersion with totally geodesic fibers is homotopically non-trivial, it follows that the $2$-dilation of
any such map is greater or equal than $2$.
\smallskip

\item An interesting problem is to classify {\em minimal maps} $f:\S^3\to\S^2$ with totally geodesic fibers. The map $f$ is
harmonic when we equip $\S^3$ with the {\em graphical metric} $g=g_{\S^3}+f^*g_{\S^2}$,
where $g_{\S^3}$ is the metric of $\S^3$ and $g_{\S^2}$ is the metric of $\S^2$. It seems that the methods
developed in the present paper cannot be easily used, since in the structure equations \eqref{ricatti} the curvature operator of the graphical metric
shows up; for more information regarding this problem we refer to \cite{savas-4}.

\smallskip
\item
In \cite{savas3} it is shown that the graphical mean curvature flow (GMCF) will smoothly deform any strictly length decreasing map
$g:\S^2\to\S^2$ into a constant map. Moreover, this process will preserve the strict length decreasing property. Hence,
GMCF will generate a smooth deformation of a unit vector field with totally geodesic fibers into a Hopf vector field.

\smallskip
\item An analogue problem in submanifold theory is to classify minimal
$m$-dimensional submanifolds in $\S^n$ with index of relative nullity at least $m-2$. Such submanifolds
are foliated by $(m-2)$-dimensional totally geodesic spheres. Under the assumption of completeness, it turns out that any
such submanifold is either totally geodesic or has dimension three. In the latter case there are plenty of examples,
even compact ones. For more details, we refer to \cite{dajczer1,savas0}.
\end{enumerate}
\end{remark}



\begin{bibdiv}
\begin{biblist}

\bib{assimos}{article}{
   author={Assimos, R.},
   author={Savas-Halilaj, A.},
   author={Smoczyk, K.},
   title={Mean curvature flow of area decreasing maps in codimension two},
   journal={arXiv: 2201.05523},
   date={2022},
   pages={1--39},
}

\bib{baird1}{book}{
   author={Baird, P.},
   author={Wood, J.},
   title={Harmonic morphisms between Riemannian manifolds},
   series={London Mathematical Society Monographs. New Series},
   volume={29},
   publisher={The Clarendon Press, Oxford University Press, Oxford},
   date={2003},
}

\bib{baird}{article}{
   author={Baird, P.},
   title={The Gauss map of a submersion},
   conference={
      title={Miniconference on Geometry and Partial Differential Equations},
   },
   book={
      publisher={The Australian National University},
      place={Centre for Mathematical Analysis, Canberra AUS},
   },
   date={1986},
   pages={8--24},
}

\bib{blair}{book}{
   author={Blair, D.E.},
   title={Riemannian geometry of contact and symplectic manifolds},
   series={Progress in Mathematics},
   volume={203},
   publisher={Birkh\"{a}user Boston, Ltd., Boston, MA},
   date={2010},
}

\bib{boeckx}{article}{
   author={Boeckx, E.},
   author={Vanhecke, L.},
   title={Harmonic and minimal vector fields on tangent and unit tangent
   bundles},
   journal={Differential Geom. Appl.},
   volume={13},
   date={2000},
   pages={77--93},
}

\bib{borrelli}{article}{
   author={Borrelli, V.},
   author={Gil-Medrano, O.},
   title={A critical radius for unit Hopf vector fields on spheres},
   journal={Math. Ann.},
   volume={334},
   date={2006},
   number={4},
   pages={731--751},
}

\bib{brito}{article}{
   author={Brito, F.},
   title={Total bending of flows with mean curvature correction},
   journal={Diff. Geom. and its Appl.},
   volume={334},
   date={2000},
   pages={157--163},
}

\bib{dajczer1}{article}{
   author={Dajczer, M.},
   author={Kasioumis, Th.},
   author={Savas-Halilaj, A.},
   author={Vlachos, Th.},
   title={Complete minimal submanifolds with nullity in Euclidean spheres},
   journal={Comment. Math. Helv.},
   volume={93},
   date={2018},
   pages={645--660},
}

\bib{dombrowski}{article}{
   author={Dombrowski, P.},
   title={On the geometry of the tangent bundle},
   journal={J. Reine Angew. Math.},
   volume={210},
   date={1962},
   pages={73--88},
}

\bib{eells}{article}{
   author={Eells, J.},
   title={Minimal graphs},
   journal={Manuscripta Math.},
   volume={28},
   date={1979},
   pages={101--108},
}

\bib{escobales}{article}{
author={Escobales, R.},
title={Riemannian submersions with totally geodesic fibers},
journal={J. Differential Geom.},
volume={10},
date={1975},
pages={253-276},
}

\bib{medrano1}{article}{
   author={Gil-Medrano, O.},
   title={Volume minimising unit vector fields on three dimensional space
   forms of positive curvature},
   journal={Calc. Var. Partial Differential Equations},
   volume={61},
   date={2022},
   number={2},
   pages={Paper No. 66, 8},
}

\bib{medrano2}{article}{
   author={Gil-Medrano, O.},
   author={Llinares-Fuster, E.},
   title={Minimal unit vector fields},
   journal={Tohoku Math. J.},
   volume={54},
   date={2002},
   pages={71--84},
}

\bib{medrano3}{article}{
   author={Gil-Medrano, O.},
   author={Llinares-Fuster, E.},
   title={Second variation of volume and energy of vector fields. Stability of Hopf vector fields},
   journal={Math. Ann.},
   volume={320},
   date={2012},
   pages={531--545},
}


\bib{gluck0}{article}{
   author={Gluck, H.},
   title={Great circle fibrations and contact structures on the 3-sphere},
   journal={arXiv:1802.03797},
   date={2018},
   pages={1--19},
}

\bib{gluck1}{article}{
   author={Gluck, H.},
   author={Gu, W.},
   title={Volume-preserving great circle flows on the 3-sphere},
   journal={Geom. Dedicata},
   volume={88},
   date={2001},
   pages={259--282},
}

\bib{gluck2}{article}{
   author={Gluck, H.},
   author={Ziller, W.},
   title={On the volume of a unit vector field on the three-sphere},
   journal={Comment. Math. Helv.},
   volume={61},
   date={1986},
   pages={177--192},
}

\bib{gluck3}{article}{
   author={Gluck, H.},
   author={Warner, F.},
   title={Great circle fibrations of the three-sphere},
   journal={Duke Math. J.},
   volume={50},
   date={1983},
   pages={107-132},
}

\bib{gluck4}{article}{
   author={Gluck, H.},
   author={Warner, F.},
   author={Ziller, W.},
   title={The geometry of the Hopf fibrations}
   journal={L'Enseign. Math.},
   volume={32},
   date={1986}
   pages={173--198},
}

\bib{han}{article}{
   author={Han, D.-S.},
   author={Yim, J-W.},
   title={Unit vector fields on spheres, which are harmonic maps}
   journal={Math. Z.},
   volume={227},
   date={1998}
   pages={83--92},
}

\bib{heller}{article}{
   author={Heller, S.},
   title={Conformal fibrations of $\Bbb S^3$ by circles},
   conference={
      title={Harmonic maps and differential geometry},
   },
   book={
      series={Contemp. Math.},
      volume={542},
      publisher={Amer. Math. Soc., Providence, RI},
   },
   date={2011},
   pages={195--202},
}

\bib{higuchi}{article}{
   author={Higuchi, A.},
   author={Kay, B.},
   author={Wood, C.M.}
   title={The energy of unit vector fields on the 3-sphere},
   journal={J. Geom. Phys.},
   volume={37},
   date={2001},
   pages={137--155},
}

\bib{johnson}{article}{
   author={Johnson, D.L.},
   title={Volumes of flows},
   journal={Proc. Am. Math. Soc.},
   volume={104},
   date={1988},
   pages={923--931},
}

\bib{savas-4}{article}{
   author={Markellos, M.},
   author={Savas-Halilaj, A.},
   title={Rigidity of the Hopf fibration},
   journal={Calc. Var. Partial Differential Equations},
   volume={60},
   date={2021},
   number={5},
   pages={Paper No. 171, 34},
}

\bib{pedersen}{article}{
   author={Pedersen, S.L.},
   title={Volumes of vector fields on spheres},
   journal={Trans. Am. Math. Soc.},
   volume={336},
   date={1993},
   pages={69--78},
}

\bib{sasaki}{article}{
   author={Sasaki, S.},
   title={On the differential geometry of tangent bundles of Riemannian manifolds},
   journal={Tohoku Math. J.},
   volume={10},
   date={1958},
   pages={338--354},
}


\bib{ruh}{article}{
   author={Ruh, E.A.},
   author={Vilms, J.},
   title={The tension field of the Gauss map},
   journal={Trans. Am. Math. Soc.},
   volume={149},
   date={1970},
   pages={569--573},
}

\bib{savas3}{article}{
   author={Savas-Halilaj, A.},
   author={Smoczyk, K.},
   title={Evolution of contractions by mean curvature flow},
   journal={Math. Ann.},
   volume={361},
   date={2015},
   pages={725--740},
}


\bib{savas1}{article}{
   author={Savas-Halilaj, A.},
   author={Smoczyk, K.},
   title={Bernstein theorems for length and area decreasing minimal maps},
   journal={Calc. Var. Partial Differential Equations},
   volume={50},
   date={2014},
   pages={549--577},
}

\bib{savas0}{article}{
   author={Savas-Halilaj, A.},
   title={On deformable minimal hypersurfaces in space forms},
   journal={J. Geom. Anal.},
   volume={23},
   date={2013},
   pages={1032--1057},
}

\bib{wangg}{article}{
   author={Wang, G.},
   title={$\S^1$-invariant harmonic maps from $\S^3$ to $\S^2$},
   journal={Bull. London Math. Soc.},
   volume={32},
   date={2000},
   pages={729--735},
  }
  
  \bib{wiegmink}{article}{
   author={Wiegmink, G.},
   title={Total bending of vector fields on Riemannian manifolds},
   journal={Math. Ann.},
   volume={303},
   date={1995},
   pages={325--344},
  }

\bib{wood}{article}{
   author={Wood, C.M.},
   title={The energy of a unit vector field},
   journal={Geom. Dedicata},
   volume={64},
   date={1997},
   pages={319--330},
}

\bib{yampolsky}{article}{
   author={Yampolsky, A.},
   title={On the mean curvature of a unit vector field},
   journal={Publ. Math.},
   volume={60},
   date={2002},
   pages={131--155},
}
		
\end{biblist}
\end{bibdiv}
\end{document}